\documentclass[12pt, a4paper, twoside]{article}
\usepackage[utf8]{inputenc}
\usepackage{amsmath}
\usepackage{amsthm}
\usepackage{amssymb}
\usepackage{graphicx}
\usepackage{xcolor}
\usepackage{subcaption}

\newtheorem{theorem}{Theorem}[section]

\newtheorem{lemma}[theorem]{Lemma}
\newtheorem{definition}[theorem]{Definition}

\graphicspath{ {./images/} }

\newcommand{\D}{\mathrm{d}}

\newcommand{\Div}{\mathrm{div}}

\title{A Tikhonov approach to level set curvature computation}
\date{\today}
\author{
Dennis Zvegincev
\thanks{Center for Industrial Mathematics,
University of Bremen,
Bremen, Germany}
}

\begin{document}
\maketitle

\begin{abstract} 
In numerical simulations of two-phase flows,
the computation of the curvature of the interface is a crucial ingredient.
Using a finite element and level set discretization, the discrete interface
is typically the level set of a low order polynomial, which often results
in a poor approximation of the interface curvature.
We present an approach to curvature computation using an approximate inversion
of the $L^2$ projection operator from the Sobolev space $H^2$ or $H^3$.
For finite element computation of the approximate inverse, the resulting higher
order equation is reformulated as a system of second order equations.
Due to the Tikhonov regularization, the method is demonstrated to be stable
against discretization irregularities. Numerical examples are shown for interior
interfaces as well as interfaces intersecting the boundary of the domain.
\end{abstract}

\section{Motivation and background}

In this paper we demonstrate a Tikhonov based approach to reconstruct the Laplacian of a function where we only know its piecewise linear representation in a typical FEM space.

With $\mathbb{P}_h^m(\Omega)$ we describe such a function space of piecewise polynomial and globally continous functions of degree $m$ over a given triangulation $\mathcal{T}_h$ of a domain $\Omega \subset \mathbb{R}^d$, $d\geq 1$. Furthermore $\mathbb{P}_h^m(\Omega)$ ist equipped with the $L^2(\Omega)$ scalar-product.

As an application for this reconstruction we consider the two-phase incompressible Navier-Stokes equation
\begin{align*}
&\rho_i \left(\frac{\partial u}{\partial t} + u \cdot \nabla u \right) = - \nabla p + \Div(\mu_i D(u)) + \rho_i g &\text{~in~} \Omega_i,\\
&\Div(u) = 0 &\text{~in~} \Omega_i,\\
&\left[\sigma \right]_\Gamma n_\Gamma = - \tau \kappa n_\Gamma,~ [u] = 0  &\text{~on~} \Gamma,\\
&V_\Gamma = n_\Gamma \cdot u  &\text{~on~} \Gamma,
\end{align*} which we solve with the extended Finite Element Method (XFEM) \cite{HANSBO20025537} using a level set approach. We will employ a piecewise linear level set function $\phi_h$ whose zero-level describes the discretized interface $\Gamma_h$.

Since the discrete interface $\Gamma_h$ is piecewise linear it is not possible to define its pointwise curvature in a meaningful setting. One possible way to express the surface tension functional lies in the Laplace-Beltrami discretization, but it is well known that convergence properties of that discretized functional is poor, i.e. for the case of a piecewise linear $\phi_h$ the error for the surface tension functional is of order $\mathcal{O}(\sqrt{h})$ \cite{gross2011numerical}.

But we can exploit the relationships between the gradient of the levelset function, the normal of the interface and the curvature of the interface, i.e. in the analytical setting we have

\begin{align*}
n_\Gamma(x) = -\frac{\nabla( \Phi_e(x))}{|\nabla(\Phi_e(x))|}, ~ \kappa(x) = \Div ( n_\Gamma(x)) \text{~for all x~} \in \Gamma,
\end{align*} where $\Phi_e$ is an exact level set function whose zero-level describes $\Gamma$. Under the assumption that $\Phi_e$ is a signed distance function to $\Gamma$ we have $|\nabla \Phi_e| = 1$ almost everywhere. Therefore the curvature $\kappa$ could simply be  calculated via 
\begin{align*}
\kappa(x) = -\Delta \Phi_e(x) \text{~for all x~} \in \Gamma.
\end{align*}

Now in the numerical setting, where we only have $\phi_h \in \mathbb{P}_h^1(\Omega)$, the Laplacian of $\phi_h$ is unfortunately of distributional type and we can't use this relationship directly. As mentioned above other methods have to be used to describe the surface tension functional, e.g. the Laplace-Beltrami discretization if one is only interested in the functional, or weak problem methods like in \cite{MARCHANDISE2007949} or improved discretization schemes like in \cite{Lervag2013} if the curvature term is of interest itself.

In this paper we will derive a smooth substitute $\Phi$ for $\phi_h$, which lies in a Sobolev-space $H^k(\Omega)$, $k \geq 2$. This way we can get a meaningful approximation to the interface curvature by calculating $\Delta \Phi$, assuming that $\phi_h$ and thus $\Phi$ are close to a signed distance function. We will also put emphasis on receiving meaningful curvature values on the boundary, i.e. where $\Gamma$ intersects $\partial \Omega$, as we're interested in accurately calculating wetting phenomena.

In order to derive such a smooth function we will invert the $L^2$-projection operator $A_m:H^k(\Omega) \to \mathbb{P}^m_h(\Omega)$.

For all standard functional analysis techniques we employ in the following sections we will refer to \cite{werner2018funktionalanalysis}.

\begin{definition}
\label{OperatorDefinition}
We define the $L^2$-projection operator $A_m:H^k(\Omega) \to \mathbb{P}^m_h(\Omega)$, $\Phi_{e} \mapsto \phi_h$ where $\phi_h \in \mathbb{P}_h^m(\Omega)$ is the solution to the following weak problem:

For a given $\Phi_e \in H^k(\Omega)$, find $\phi_h \in \mathbb{P}^m_h(\Omega)$ such that
\begin{align}
\label{1Laplace}
\int_\Omega \phi_h \psi_h ~\D x = \int_\Omega \Phi_e \psi_h ~\D x
\end{align} for all $\psi_h \in \mathbb{P}^m_h(\Omega)$.
\end{definition}

This operator comes in play whenever we take a smooth and exact level-set function $\Phi_e$ and project it onto a discrete level-set function $A_m(\Phi_e) =:\phi_h \in \mathbb{P}^m_h(\Omega)$. Therefore, it is natural to try to invert this process and try to calculate a smooth preimage under the operator $A_m$ for a given discrete level-set function.

\begin{lemma}[Properties of $A_m$]
\label{LemmaProperties}
The following assertions for the operator $A_m$ hold true:

\begin{itemize}
\item[i)] The operator $A_m$ is well defined, i.e. for every $\Phi_e \in H^k(\Omega)$ there exists a unique function $\phi_h = A_m\Phi_e \in \mathbb{P}^m_h(\Omega)$ such that \eqref{1Laplace} is fulfilled. Furthermore, $A_m$ is a linear operator.
\item[ii)] For every $\Phi \in H^k(\Omega)$ the operator $A_m$ fulfills the equation $$ \langle A_m\Phi, A_m\Phi \rangle_{L^2(\Omega)} = \langle\Phi, A_m\Phi \rangle_{L^2(\Omega)}.$$
\item[iii)] For every $\Phi \in H^k(\Omega)$ the inequation $$\| A_m\Phi\|_{L^2(\Omega)} \leq \|\Phi\|_{L^2(\Omega)}$$ holds true.
\end{itemize}
\end{lemma}

\begin{proof}~\\
\textbf{i):} First we will show the operator $A_m$ is well-defined. The operator is defined by the solution of the variational formulation in \eqref{1Laplace} and therefore it suffices to argue that the left-hand side constitutes the $L^2(\Omega)$-scalar product, whereas the right hand side is a continous linear functional on $\mathbb{P}^m_h(\Omega)$. In our definition $\mathbb{P}^m_h(\Omega)$ is equipped with the $L^2(\Omega)$-scalar product and therefore the Riesz-representation theorem guarantees a unique solution $\phi_h \in \mathbb{P}^m_h(\Omega)$ for every $\Phi_e \in H^k(\Omega)$.
As $A$ is the solution operator from the Riesz-representation theorem we know that the operator in question is thus linear.
\\\\
\textbf{ii):}
Let $\Phi \in H^k(\Omega)$ be abritrary. Recalling the definition of the operator $A_m$ in Definition \ref{OperatorDefinition}, we directly get
\begin{align*}
\langle A_m \Phi, A_m \Phi \rangle_{L^2(\Omega)}
=
\int_\Omega A_m \Phi \underbrace{A_m \Phi}_{\in \mathbb{P}^m_h(\Omega)} \D x
=
\int_\Omega \Phi \underbrace{A_m \Phi}_{\in \mathbb{P}^m_h(\Omega)} \D x
=
\langle \Phi, A_m \Phi \rangle_{L^2(\Omega)}
\end{align*} as (the left) $A_m \Phi$ is the solution to the variational problem in \eqref{1Laplace} to the correspoding $\Phi$ and therefore fulfills \eqref{1Laplace} while (the right) $A_m \Phi$ functions as a testfunction in $\mathbb{P}^m_h(\Omega)$.
\\\\
\textbf{iii):}
This property is a direct result of ii). Let $\Phi \in H^k(\Omega)$ be abritrary. Then we have
\begin{align*}
\|A_m\Phi\|^2_{L^2(\Omega)} = \langle A_m \Phi, A_m \Phi \rangle_{L^2(\Omega)}
\overset{ii)}{=}
\langle A_m\Phi, \Phi \rangle_{L^2(\Omega)}
\overset{C.S.}{\leq}
\|A_m \Phi\|_{L^2(\Omega)}\| \Phi\|_{L^2(\Omega)}
\end{align*}where we employ the Cauchy-Schwartz inequality (C.S.). 
\end{proof}

In this paper we will specifically look at the operator $A_1$, i.e. the projection into the piecewise linear FEM space $\mathbb{P}^1_h(\Omega)$. Furthermore we will omit the subscript 1 from now on and refer to our problem operator simply as $A$.

Looking at the function spaces it is obvious that the inversion of $A$ is an ill-posed problem, as the operator $A$ maps from an infinite dimensional vector space onto a finite one, therefore necessitating a regularization approach.
\\\\
In Section 2 we will define a Tikhonov functional to find a suitable solution to this ill-posed problem and derive a higher order PDE whose solution coincides with the minimum of the Tikhonov functional.
In Section 3 we will then reformulate this higher order PDE into a system of 2nd order PDEs followed by Section 4 with numerical experiments for our reconstruction. This paper is then closed by the conclusion and an outlook in Section 5.

\section{Analytical Problem formulation}
We want to find a smooth function $\Phi \in H^k(\Omega)$ whose image under $A$ is close to $\phi_h$. As described above, inversion of $A$ is ill-posed and thus it is necessary to find a regularized solution. 

In order to achieve our task we will choose the classical Tikhonov-functional
\begin{align}
\label{Functional}
J(\Phi) := \|A\Phi - \phi_h \|_{L^2(\Omega)}^2 + \alpha \| \Phi \|^2_{H^k(\Omega)}.
\end{align}

\begin{theorem}
\label{TheoremUniqueness}
For every $\phi_h \in \mathbb{P}^1_h(\Omega)$ there exists a unique global minimum $\Phi \in H^k(\Omega)$ of the functional $J$ defined in \eqref{Functional}.
\end{theorem}

\begin{proof}
In order to show existence and uniqueness of a global minizer of \eqref{Functional}, we first note that we minize over the whole Hilbert-space of $H^k(\Omega)$. Looking at our regularizer, it is of the form $W(\Phi):=\alpha \|\Phi\|^2_{H^k(\Omega)}$. This regularizer $W$ is obviously bounded from below with $W(\Phi) = \alpha\|\Phi\|^2_{H^k(\Omega)} \geq 0$ for all  $\Phi\in H^k(\Omega)$.

Let $\{x_n\}_{n\in \mathbb{N}} \subset H^k(\Omega)$ be a $W$-bounded sequence, i.e. there exists a $k > 0$ such that $|W(x_n)| \leq k$ for all $n\in \mathbb{N}$. Therefore $\{x_n\}_{n\in \mathbb{N}}$ is also a bounded sequence in $H^k(\Omega)$ as with our choice of $W$ we have
\begin{align*}
|W(\Phi)| =\alpha \|\Phi\|^2_{H^k(\Omega)}  \leq k.
\end{align*}
Since $H^k(\Omega)$ is a Hilbert-space, a result from the Banach-Alaoglu theorem states that every bounded sequence in $H^k(\Omega)$ and therefore every $W$-bounded sequence has a weakly convergent subsequence i.e. there exists an $x\in H^k(\Omega)$ such that $x_{n_j} \rightharpoonup x$.

Furthermore, norms in reflexive Banach-spaces are known to be weakly lower-semicontinous. Now let $\{x_n\}_{n \in \mathbb{N}}$ be a $W$-bounded sequence which weakly converges towards $x \in H^k(\Omega)$. Therefore there exists a subsequence  which also fulfills $W(x) \leq \liminf_{j \to \infty} W(x_{n_j})$ as $W$ is simply a squared norm of the sequence space.

Last but not least, since $W$ consists of a squared norm of the space $H^k(\Omega)$ it is therefore a strictly convex functional over $H^k(\Omega)$.

We therefore fulfill all requirements of Theorem 2.5 in \cite{MAZZIERI2012396} and thus our functional $J$ in \eqref{Functional} has a unique and global minimizer.
\end{proof}

In order to calculate the minimum of \eqref{Functional}, we will now derive a partial differential equation whose solution is the same as the minimum of \eqref{Functional}. Optimality condition for the minimum of $J$ yields
\begin{align}
\label{NecCondition}
\delta J(\Phi,\Psi) = 0
\end{align} for every $\Psi \in H^k(\Omega)$, where $\delta J(\Phi,\Psi)$ is the Gateaux-derivative of $J$ at the point $\Phi$ in direction $\Psi$.

Explicitly calculating this Gateaux-derivative gives us
\begin{align*}
\delta J(\Phi,\Psi) = \lim_{\epsilon \to 0} \frac{J(\Phi + \epsilon \Psi) - J(\Phi)}{\epsilon}\\
=
\lim_{\epsilon \to 0} \frac{1}{\epsilon}
\left(\|A(\Phi+\epsilon \Psi) -  \phi_h\|^2_{L^2(\Omega)} + \alpha \|\Phi + \epsilon \Psi\|_{H^k(\Omega)} 
-
\|A\Phi -  \phi_h\|^2_{L^2(\Omega)} + \alpha \|\Phi \|_{H^k(\Omega)}  \right)\\
=
\lim_{\epsilon \to 0} \frac{1}{\epsilon}
\big(
\langle A\Phi + A \epsilon \Psi - \phi_h,  A\Phi + A \epsilon \Psi  - \phi_h \rangle_{L^2(\Omega)}
+ \alpha \langle \Phi + \epsilon \Psi, \Phi + \epsilon \Psi \rangle_{H^k(\Omega)}\\
-\langle A\Phi  - \phi_h,  A\Phi  - \phi_h  \rangle_{L^2(\Omega)}
- \alpha \langle \Phi , \Phi   \rangle_{H^k(\Omega)}
\big)\\
=
\lim_{\epsilon \to 0} \frac{1}{\epsilon}
\big(
2\epsilon \langle A\Phi  - \phi_h, A\Psi    \rangle_{L^2(\Omega)} + \epsilon^2 \langle A\Psi, A\Psi  \rangle_{L^2(\Omega)}  + 2\alpha \epsilon \langle \Phi,\Psi\rangle_{H^k(\Omega)} + \epsilon^2 \langle \Psi, \Psi \rangle_{H^k(\Omega)}
\big)\\
=
2 \langle A\Phi  - \phi_h, A\Psi    \rangle_{L^2(\Omega)}
+
 2\alpha  \langle \Phi,\Psi\rangle_{H^k(\Omega)}.
\end{align*}
With the necessary condition \eqref{NecCondition} we get the following weak problem formulation: Find $\Phi \in H^k(\Omega)$ such that
\begin{align}
\label{PDEFirstVersion}
  \langle A\Phi , A\Psi    \rangle_{L^2(\Omega)}
+
 \alpha  \langle \Phi,\Psi\rangle_{H^k(\Omega)}
 =
   \langle \phi_h , \Psi    \rangle_{L^2(\Omega)}
\end{align} for all $\Psi\in H^k(\Omega)$.

\begin{lemma}
\label{LemUniquenessPDEFirst}
For every $\phi_h \in \mathbb{P}^m_h(\Omega)$ there exists a unique solution $\Phi \in H^k(\Omega)$ such that \eqref{PDEFirstVersion} is fulfilled for all $\Psi\in H^k(\Omega)$.
\end{lemma}

\begin{proof}
Showing unique existence of the solution is a classical application of the Lax-Milgram theorem. First, for the RHS we have
\begin{align*}
\langle \phi_h, \Psi \rangle_{L^2(\Omega)} \overset{C.S.}{\leq} \|\phi_h\|_{L^2(\Omega)}\|\Psi\|_{L^2(\Omega)} 
\leq
\|\phi_h\|_{L^2(\Omega)}\|\Psi\|_{H^k(\Omega)}
\end{align*} for every $\Psi \in H^k(\Omega)$. As we have $\phi_h \in  \mathbb{P}^1_h(\Omega)$, the function $\phi_h$ itself and therefore its $L^2(\Omega)$-norm are bounded and thus the RHS is a linear continous functional on $H^k(\Omega)$.

Similarly, for the LHS we receive
\begin{align*}
  \langle A\Phi , A\Psi    \rangle_{L^2(\Omega)}
+
 \alpha  \langle \Phi,\Psi\rangle_{H^k(\Omega)}
 \overset{C.S.}{\leq}
 \|A\Phi\|_{L^2(\Omega)} \|A\Psi\|_{L^2(\Omega)}
 +
 \alpha\|\Phi\|_{H^k(\Omega)}\|\Psi\|_{H^k(\Omega)}\\
  \overset{Lem. ~\ref{LemmaProperties} ~iii)}{\leq}
   \|\Phi\|_{L^2(\Omega)} \|\Psi\|_{L^2(\Omega)}
 +
 \alpha\|\Phi\|_{H^k(\Omega)}\|\Psi\|_{H^k(\Omega)}\\
 \leq 
(1+\alpha) \|\Phi\|_{H^k(\Omega)}\|\Psi\|_{H^k(\Omega)}
\end{align*} for every $\Phi, \Psi \in H^k(\Omega)$
and
\begin{align*}
  \langle A\Phi , A\Phi    \rangle_{L^2(\Omega)}
+
 \alpha  \langle \Phi,\Phi\rangle_{H^k(\Omega)}
 \geq
 \alpha \langle \Phi,\Phi\rangle_{H^k(\Omega)}
 = \alpha\|\Phi\|_{H^k(\Omega)}^2
\end{align*} for every $\Phi \in H^k(\Omega)$.
The bilinearform on the LHS of \eqref{PDEFirstVersion} is therefore continous and coercive. 

As such our weak formulation in \eqref{PDEFirstVersion} fulfills all requirements of the Lax-Milgram theorem and for every $\phi_h$ we receive a unique solution $\Phi \in H^k(\Omega)$.
\end{proof}

So far we have established in Theorem \ref{TheoremUniqueness} that the minimum of our functional $J$ is unique while the weak form PDE in \eqref{PDEFirstVersion} also has a unique solution. As this PDE was constructed via the optimality condition \eqref{NecCondition}, its solution is therefore also the desired minimum of $J$.

In order to solve this PDE numerically, we will need to eliminate the operator $A$ in the problem formulation. First, employing Lemma \ref{LemmaProperties} ii), we can simplify the formulation in \eqref{PDEFirstVersion} as
\begin{align*}
  \langle A\Phi , A\Psi    \rangle_{L^2(\Omega)}
+
 \alpha  \langle \Phi,\Psi\rangle_{H^k(\Omega)}
 =
   \langle \phi_h , \Psi    \rangle_{L^2(\Omega)}\\
   \Leftrightarrow
     \langle A\Phi , \Psi    \rangle_{L^2(\Omega)}
+
 \alpha  \langle \Phi,\Psi\rangle_{H^k(\Omega)}
 =
   \langle \phi_h , \Psi    \rangle_{L^2(\Omega)}.
\end{align*}

Now we will replace $A\Phi$ with another unknown $\lambda \in \mathbb{P}_h^1(\Omega)$ which will be part of the solution. Obviously $A\Phi =  \lambda$ will need to be satisfied as well and recalling the definition \ref{OperatorDefinition} of the projection operator $A$ means the following two integrals
\begin{align*}
\int_\Omega \lambda \eta \D x = \int_\Omega \Phi \eta \D x
\end{align*} will need to match for every $\eta \in \mathbb{P}^1_h(\Omega)$. 

As such we receive the following mixed problem formulation: 
Find $(\Phi,\lambda) \in H^k(\Omega) \times \mathbb{P}^1_h(\Omega)$ such that
\begin{equation}
\label{PDEMixed}
\begin{split}
 \alpha  \langle \Phi,\Psi\rangle_{H^k(\Omega)}
 +
   \langle \lambda , \Psi    \rangle_{L^2(\Omega)}
 &=
   \langle \phi_h , \Psi    \rangle_{L^2(\Omega)}\\
   \langle \lambda, \eta  \rangle_{L^2(\Omega)} - \langle \Phi, \eta \rangle_{L^2(\Omega)}&= 0
\end{split}
\end{equation} for every $(\Psi, \eta) \in H^k(\Omega) \times \mathbb{P}^1_h(\Omega)$.

It is important to keep in mind that this isn't a new problem formulation for which we would need to show existence and uniqueness of a solution. Rather this is the formulation in equation \eqref{PDEFirstVersion} where we have written out the definition of the operator $A$. As such we already know a unique solution pair $(\Phi,\lambda) \in H^k(\Omega) \times \mathbb{P}^1_h(\Omega)$ exists while $\Phi$ minimizes our functional $J$.

\subsection{Choice of Sobolev-space $H^k(\Omega)$ and strong problem formulation}
\label{SectionChoiceSobolev}
Up to this point we have not yet explicitly chosen the degree of weak-differentialibility we want from the smooth reconstruction. Since we're interested in the curvature of the interface, we have to set atleast $k \geq 2$.

Different degrees for $k$ will result in different boundary conditions our reconstruction will need to fulfill and we will examine the possible choices $k =2,3$.

As a first step, we will perform integration by parts on the various terms of the scalar product $\langle \cdot, \cdot \rangle_{H^3(\Omega)}$. As a slight modification we will also introduce a different regularization parameter $\alpha_k$ for each term. 

Additionally we will assume $H^6(\Omega)$-regularity of $\Phi$, so that we can derive a strong problem formulation. Thus we get the four terms $B_0, \hdots, B_3$ as
\begin{align*}
B_0(\Phi, \Psi) := \alpha_0 \langle \Phi, \Psi \rangle_{L^2(\Omega)},
\end{align*}
\begin{align*}
B_1(\Phi, \Psi) :=& \alpha_1 \sum_{i} \langle \partial_i \Phi, \partial_i \Psi  \rangle_{L^2(\Omega)}
=
-\alpha_1 \sum_{i}  \langle \partial_{i}  \partial_{i} \Phi, \Psi \rangle_{L^2(\Omega)}
+
\alpha_1\int_{\partial \Omega}  \nu^T \nabla \Phi \Psi  ~\D S\\
=&
-\alpha_1 \langle \Delta \Phi, \Psi \rangle_{L^2(\Omega)}
+
\alpha_1\int_{\partial \Omega}  \nu^T \nabla \Phi \Psi  ~\D S,
\end{align*}
\begin{align*}
B_2(\Phi, \Psi):=& \alpha_2 \sum_{i,j} \langle \partial_i \partial_j \Phi,\partial_i \partial_j \Psi \rangle_{L^2(\Omega)} 
=
-\alpha_2 \sum_{i,j} \langle \partial_i \partial_j \partial_i \Phi, \partial_j \Psi  \rangle_{L^2(\Omega)} 
+ 
\alpha_2\int_{\partial\Omega} \nu^T  H_{\Phi} \nabla \Psi  ~\D S\\
=&
\alpha_2 \sum_{i,j} \langle \partial_j\partial_i\partial_j\partial_i \Phi, \Psi \rangle_{L^2(\Omega)}
-
\alpha_2 \sum_{i,j} \int_{\partial \Omega} \nu_j \partial_i\partial_j\partial_i \Phi  \Psi ~\D S
+ 
\alpha_2\int_{\partial\Omega} \nu^T  H_{\Phi} \nabla \Psi  ~\D S\\
=&
\alpha_2 \sum_{i,j} \langle \partial_j\partial_j\partial_i\partial_i \Phi, \Psi \rangle_{L^2(\Omega)}
-
\alpha_2 \sum_{i,j} \int_{\partial \Omega} \nu_j \partial_j\partial_i\partial_i \Phi  \Psi ~\D S
+ 
\alpha_2\int_{\partial\Omega} \nu^T  H_{\Phi} \nabla \Psi  ~\D S\\
=&
\alpha_2 \langle \Delta \Delta \Phi, \Psi \rangle_{L^2(\Omega)}
-
\alpha_2 \int_{\partial \Omega} \nu^T \nabla(\Delta \Phi) \Psi ~\D S
+ 
\alpha_2\int_{\partial\Omega} \nu^T  H_{\Phi} \nabla \Psi  ~\D S,
\end{align*}
\begin{align*}
B_3(\Phi, \Psi) :=& \alpha_3 \sum_{i,j,l} \langle \partial_i\partial_j\partial_l\Phi,\partial_i\partial_j\partial_l \Psi   \rangle_{L^2(\Omega)}\\
=&
- \alpha_3 \sum_{i,j,l} \langle \partial_i\partial_i\partial_j\partial_l\Phi,\partial_j\partial_l \Psi   \rangle_{L^2(\Omega)} 
+
\alpha_3 \sum_{i,j,l} \int_{\partial\Omega} \nu_i\partial_i\partial_j\partial_l\Phi \partial_j\partial_l \Psi  ~\D S\\
=&
 \alpha_3 \sum_{i,j,l} \langle \partial_j\partial_i\partial_i\partial_j\partial_l\Phi,\partial_l \Psi   \rangle_{L^2(\Omega)} 
 -
 \alpha_3 \sum_{i,j,l}\int_{\partial\Omega} \nu_j \partial_i\partial_i\partial_j\partial_l \Phi \partial_l \Psi ~\D S\\
 &+
\alpha_3 \sum_{i,j,l} \int_{\partial\Omega} \nu_i\partial_i\partial_j\partial_l\Phi \partial_j\partial_l \Psi  ~\D S\\
=&
- \alpha_3 \sum_{i,j,l} \langle\partial_l \partial_j\partial_i\partial_i\partial_j\partial_l\Phi, \Psi   \rangle_{L^2(\Omega)} 
+
\alpha_3 \sum_{i,j,l}\int_{\partial\Omega} \nu_l \partial_j\partial_i\partial_i\partial_j\partial_l\Phi,  \Psi ~\D S\\
 &-
 \alpha_3 \sum_{i,j,l}\int_{\partial\Omega} \nu_j \partial_i\partial_i\partial_j\partial_l \Phi \partial_l \Psi ~\D S
 +
\alpha_3 \sum_{i,j,l} \int_{\partial\Omega} \nu_i\partial_i\partial_j\partial_l\Phi \partial_j\partial_l \Psi  ~\D S\\
=&
- \alpha_3 \sum_{i,j,l} \langle\partial_l \partial_l\partial_j\partial_j\partial_i\partial_i\Phi, \Psi   \rangle_{L^2(\Omega)} 
+
\alpha_3 \sum_{i,j,l}\int_{\partial\Omega} \nu_l \partial_l\partial_j\partial_j\partial_i\partial_i\Phi,  \Psi ~\D S\\
 &-
 \alpha_3 \sum_{i,j,l}\int_{\partial\Omega} \nu_j \partial_l\partial_j\partial_i\partial_i \Phi \partial_l \Psi ~\D S
 +
\alpha_3 \sum_{i,j,l} \int_{\partial\Omega} \nu_i\partial_i\partial_j\partial_l\Phi \partial_j\partial_l \Psi  ~\D S\\
=&
-\alpha_3 \langle \Delta\Delta\Delta \Phi, \Psi \rangle_{L^2(\Omega)}
+
\alpha_3 \int_{\partial\Omega} \nu^T \nabla(\Delta\Delta \Phi)\Psi  ~\D S\\
&-
\alpha_3 \int_{\partial\Omega} \nu^T H_{\Delta \Phi} \nabla \Psi ~\D S
+
\alpha_3\int_{\partial\Omega} \nu^T \nabla(H_{\Phi}) : H_\Psi ~\D S,
\end{align*}
where $\nu$ denotes the outer normal vector on $\partial \Omega$ and $H_\Phi$ the Hessian of $\Phi$.

Now we can identify the boundary conditions and strong problem formulation by summing the appropriate parts and sorting the boundary integrals by the derivates of the test function $\Psi$.

\subsubsection{Choice of $k=2$}
In this case we sum $B_0$ to $B_2$ and get
\begin{align*}
&B_0(\Phi,\Psi) + B_1(\Phi,\Psi) +B_2(\Phi,\Psi)\\
=&
\alpha_0 \langle \Phi, \Psi \rangle_{L^2(\Omega)}
-
\alpha_1 \langle \Delta \Phi, \Psi \rangle_{L^2(\Omega)}
+
\alpha_2 \langle \Delta \Delta \Phi, \Psi \rangle_{L^2(\Omega)}\\
&+
\alpha_1\int_{\partial \Omega}  \nu^T \nabla \Phi \Psi  ~\D S
-
\alpha_2 \int_{\partial \Omega} \nu^T \nabla(\Delta \Phi) \Psi ~\D S\\
&+ 
\alpha_2\int_{\partial\Omega} \nu^T  H_{\Phi} \nabla \Psi  ~\D S\\
=&
\langle \alpha_0 \Phi - \alpha_1 \Delta \Phi + \alpha_2 \Delta \Delta \Phi, \Psi \rangle_{L^2(\Omega)}\\
&+
\int_{\Omega} \nu^T (\alpha_1\nabla \Phi  - \alpha_2 \nabla(\Delta \Phi)  )\Psi ~\D S\\
&+ 
\alpha_2\int_{\partial\Omega} \nu^T  H_{\Phi} \nabla \Psi  ~\D S.
\end{align*}
Now when we pose the following weak problem: Find $\Phi \in H^2(\Omega)$ such that
\begin{align*}
B_0(\Phi, \Psi) + B_1(\Phi, \Psi) +B_2(\Phi, \Psi) = \langle \phi_h, \Psi \rangle_{L^2(\Omega)}
\end{align*} for all $\Psi \in H^2(\Omega)$, we can identify the corresponding strong problem as: Solve for $\Phi \in H^4(\Omega)$ such that
\begin{align*}
&\alpha_0 \Phi - \alpha_1 \Delta \Phi + \alpha_2 \Delta \Delta \Phi = \phi_h &\text{~in~} \Omega,\\
&\nu^T (\alpha_1\nabla \Phi  - \alpha_2 \nabla(\Delta \Phi)  ) = 0 &\text{~on~}\partial\Omega ,\\
&\alpha_2\nu^T H_\Phi = 0 &\text{~on~}\partial\Omega.
\end{align*} 
Unfortunately, this type of boundary condition is problematic for our method as we seek to find meaningful values for the second derivates of $\Phi$. Such a boundary condition is thus contrary to our goal as it directly sets parts of the second derivates to 0. As we will see later in the numerical experiments, this boundary condition drastically worsens the curvature expressions towards the boundary.

\subsubsection{Choice of $k=3$}
And here we sum $B_0$ to $B_3$ and receive
\begin{equation}
\label{BilinearFormH3}
\begin{split}
&B_0(\Phi,\Psi) + B_1(\Phi,\Psi) +B_2(\Phi,\Psi) + B_3(\Phi,\Psi)\\
=&
\alpha_0 \langle \Phi, \Psi \rangle_{L^2(\Omega)}
-
\alpha_1 \langle \Delta \Phi, \Psi \rangle_{L^2(\Omega)}
+
\alpha_2 \langle \Delta \Delta \Phi, \Psi \rangle_{L^2(\Omega)}
-
\alpha_3 \langle \Delta\Delta \Delta \Phi, \Psi \rangle_{L^2(\Omega)}\\
&+
\alpha_1\int_{\partial \Omega}  \nu^T \nabla \Phi \Psi  ~\D S
-
\alpha_2 \int_{\partial \Omega} \nu^T \nabla(\Delta \Phi) \Psi ~\D S\\
&+ 
\alpha_2\int_{\partial\Omega} \nu^T  H_{\Phi} \nabla \Psi  ~\D S
+
\alpha_3 \int_{\partial\Omega} \nu^T \nabla(\Delta\Delta \Phi)\Psi  ~\D S\\
&-
\alpha_3 \int_{\partial\Omega} \nu^T H_{\Delta \Phi} \nabla \Psi ~\D S
+
\alpha_3\int_{\partial\Omega} \nu^T \nabla(H_{\Phi}) : H_\Psi ~\D S\\
=&
\langle \alpha_0 \Phi - \alpha_1 \Delta \Phi + \alpha_2 \Delta \Delta \Phi -   \alpha_3 \Delta \Delta \Delta \Phi, \Psi \rangle_{L^2(\Omega)}\\
&+
\int_{\partial\Omega}  \nu^T(\alpha_1 \nabla \Phi - \alpha_2 \nabla(\Delta \Phi) + \alpha_3 \nabla(\Delta \Delta\Phi ) )\Psi ~\D S\\
&+
\int_{\partial\Omega} \nu^T(\alpha_2 H_{\Phi} - \alpha_3 H_{\Delta \Phi} )\nabla \Psi  ~\D S\\
&+
\alpha_3\int_{\partial\Omega} \nu^T \nabla(H_{\Phi}) : H_\Psi ~\D S.
\end{split}
\end{equation}
In this case the weak problem is stated as: Find $\Phi\in H^3(\Omega)$ such that
\begin{align*}
B_0(\Phi, \Psi) + B_1(\Phi, \Psi) + B_2(\Phi, \Psi) + B_3(\Phi, \Psi)
= \langle \phi_h, \Psi \rangle_{L^2(\Omega)}
\end{align*} for all $\Psi \in H^3(\Omega)$ and the corresponding strong formulation is now: Find $\Phi \in H^6(\Omega)$ such that
\begin{equation}
\label{StrongPDE}
\begin{split}
 \alpha_0 \Phi - \alpha_1 \Delta \Phi + \alpha_2 \Delta \Delta \Phi -   \alpha_3 \Delta \Delta \Delta \Phi = \phi_h \text{~in~} \Omega,\\
  \nu^T \nabla(H_{\Phi})  = 0 \text{~on~} \partial\Omega,\\
\nu^T(\alpha_2 H_{\Phi} - \alpha_3 H_{ \Phi} ) = 0 \text{~on~} \partial\Omega,\\
\nu^T(\alpha_1 \nabla \Phi - \alpha_2 \nabla \Phi + \alpha_3 \nabla\Phi ) = 0 \text{~on~} \partial\Omega.
\end{split}
\end{equation}

While we do not have a problematic boundary condition, reconstruction into $H^3(\Omega)$ will unfortunately have other woes: The signed distance level set function for a given exact interface will generally have kinks and the gradient will either have singularities or even jumps. While these problematic regions will not be on $\Gamma$ itself, if $\Gamma$ is smooth enough, it poses a different problem as the exact level set function will therefore not be a member of $H^3(\Omega)$ or sometimes not even of $H^2(\Omega)$ in the first place. An $H^3(\Omega)$ function can therefore often haven too much regularity and the reconstruction will have to oscillate near such problematic regions. A simple example would be the level set function $\phi(x):= R - |x|$ which is the level set function of a circle with radius $R$ and centered around the 0-point. For this function the gradient is not determined at 0 and its Laplacian converges to $-\infty$ around the 0-point.

\section{Discretization}
In the previous section we have seen that we can minimize the functional $J$ in \eqref{Functional} by solving a higher order PDE. In this chapter, we will now focus on how this higher order PDE can be solved numerically.

We will perform this only for the case of $H^3$ reconstruction. In the case of the $H^2$ reconstruction, all the steps are performed analogous and are left for the reader.

\subsection{Reformulation from higher order to a system of 2nd order }

Assume one want to solve the strong problem formulation in \eqref{StrongPDE}
with Finite Element Method. The standard FEM ansatz spaces are only a subset of $H^1(\Omega)$ while the weak formulation of this problem will require at least $H^3(\Omega)$ functions. 

In order to solve this numerically, we will reformulate this PDE of 6th order into a system of three PDEs of 2nd order. This is one of the possible techniques to handle higher order PDEs as demonstrated in \cite{Droniu2017} and \cite{HighOrderPDEDiss}. This gives us the benefit of directly receiving the Laplacian as part of the solution.
Alternatively one could also employ the Discontinous Galerkin scheme, as for example done in \cite{Cheng2008}, but this would require calculating the Laplacian by hand afterwards.
 For our approach we define
\begin{align*}
\Phi_1 :=& \Phi,\\
\Phi_2 :=& \Delta \Phi = \Delta \Phi_1,\\
\Phi_3 :=& \Delta \Delta \Phi = \Delta \Phi_2.
\end{align*}
Inserting these definitions into the PDE yields the system
\begin{align}
\label{PDE1}
\alpha_0 \Phi_1 - \alpha_1 \Delta \Phi_1 + \alpha_2 \Delta \Phi_2 - \alpha_3 \Delta \Phi_3 = \phi_h \text{~in~} \Omega,\\
\label{PDE2}
\Phi_2 - \Delta \Phi_1 = 0 \text{~in~} \Omega,\\
\label{PDE3}
\Phi_3 - \Delta \Phi_2 = 0 \text{~in~} \Omega,\\
\label{BC1}
  \nu^T \nabla(H_{\Phi_1})  = 0 \text{~on~} \partial\Omega,\\
  \label{BC2}
\nu^T(\alpha_2 H_{\Phi_1} - \alpha_3 H_{ \Phi_2} ) = 0 \text{~on~} \partial\Omega,\\
\label{BC3}
\nu^T(\alpha_1 \nabla \Phi_1 - \alpha_2 \nabla \Phi_2 + \alpha_3 \nabla\Phi_3 ) = 0 \text{~on~} \partial\Omega 
\end{align}

In the next step, we will multiply with testfunctions $(\Psi_1, \Psi_2, \Psi_3) \in H^3(\Omega) \times H^2(\Omega) \times H^1(\Omega)$, integrate over $\Omega$ and perform integration by parts on these equation to get a weak formulation.

Starting with \eqref{PDE1}, we receive

\begin{equation}
\begin{split}
\int_\Omega \left(\alpha_0 \Phi_1 - \alpha_1 \Delta \Phi_1 + \alpha_2 \Delta \Phi_2 - \alpha_3 \Delta \Phi_3\right) \Psi_1 \D x 
=
\int_{\Omega} \phi_h \Psi_1 \D  x\\
\begin{split}
\end{split}
 \label{WeakForm1}
\Leftrightarrow
\int_\Omega \alpha_0 \Phi_1 \Psi_1 + \alpha_1 \nabla\Phi_1 \cdot \nabla \Psi_1 - \alpha_2 \nabla \Phi_2 \cdot \nabla \Psi_1 + \alpha_3 \nabla \Phi_3 \cdot \nabla \Psi_1 \D x\\
- 
\underbrace{\int_{\partial\Omega} \nu^T(\alpha_1 \nabla \Phi_1 - \alpha_2 \nabla \Phi_2 + \alpha_3 \nabla\Phi_3 )\Psi_1 \D S}_{=0}
=
\int_{\Omega} \phi_h \Psi_1 \D  x,
 \end{split}
\end{equation} where we can directly incoorporate the boundary condition \eqref{BC3} as a natural boundary condition.

Handling the equations \eqref{PDE2} and \eqref{PDE3} in the same vein, we simply get
\begin{equation}
\label{WeakForm2}
\begin{split}
\int_\Omega (\Phi_2 -\Delta \Phi_1) \Psi_2  \D x = 0\\
\Leftrightarrow
\int_\Omega \Phi_2 \Psi_2 + \nabla \Phi_1 \cdot \nabla \Psi_2 \D x
-
\int_{\partial \Omega}  \left(\nu^T \nabla \Phi_1\right) \Psi_2  \D S = 0
\end{split}
\end{equation} and
\begin{equation}
\label{WeakForm3}
\begin{split}
\int_\Omega (\Phi_3 -\Delta \Phi_2) \Psi_3  \D x = 0\\
\Leftrightarrow
\int_\Omega \Phi_3 \Psi_3 + \nabla \Phi_2 \cdot \nabla \Psi_3 \D x
-
\int_{\partial \Omega}  \left(\nu^T \nabla \Phi_2\right) \Psi_3  \D S = 0.
\end{split}
\end{equation}
We see that we can't directly plug in the boundary conditions \eqref{BC1} and \eqref{BC2} into any of these weak forms. Now it's important to remember we arrived at the strong formulation in \eqref{StrongPDE} by performing integration by parts on the bilinear form in \eqref{BilinearFormH3}. As such we still have leftover integrals for the boundary conditions \eqref{BC2} and \eqref{BC3}. Keeping these integrals in mind when handling the bilinear form \eqref{WeakForm1} and adding the weak forms \eqref{WeakForm2} and \eqref{WeakForm3} we arrive at the problem formulation:
 Find $(\Phi_1, \Phi_2, \Phi_3) \in H^3(\Omega) \times H^2(\Omega) \times H^1(\Omega)$ such that
\begin{align*}
\int_\Omega \alpha_0 \Phi_1 \Psi_1 + \alpha_1 \nabla\Phi_1 \cdot \nabla \Psi_1 - \alpha_2 \nabla \Phi_2 \cdot \nabla \Psi_1 + \alpha_3 \nabla \Phi_3 \cdot \nabla \Psi_1 \D x\\
+
\int_\Omega \Phi_2 \Psi_2 + \nabla \Phi_1 \cdot \nabla \Psi_2 \D x
-
\int_{\partial \Omega}  \left(\nu^T \nabla \Phi_1\right) \Psi_2  \D S\\
+
\int_\Omega \Phi_3 \Psi_3 + \nabla \Phi_2 \cdot \nabla \Psi_3 \D x
-
\int_{\partial \Omega}  \left(\nu^T \nabla \Phi_2\right) \Psi_3  \D S\\
+
 \alpha_3\int_{\partial\Omega} (\nu^T \nabla(H_{\Phi_1})) : H_{\Psi_1}   \D S\\
 +
 \int_{\partial\Omega} (\nu^T(\alpha_2 H_{\Phi_1} - \alpha_3 H_{ \Phi_2} )  \nabla \Psi_1   \D S
=
\int_{\Omega} \phi_h \Psi_1 \D x
\end{align*} for all $(\Psi_1, \Psi_2, \Psi_3) \in H^3(\Omega)\times H^2(\Omega)\times H^1(\Omega)$.

\subsubsection{Complete FEM formulation for the case of $k = 3$}
\label{SectionH3Formulation}
At this point, we need to mention a final modification of our functional $J$ in \eqref{Functional}: In the regularizer $W(\Phi)$, we exchange the term $\alpha_0 \|\Phi\|^2_{L^2(\Omega)}$ with $\alpha_0 \|\Phi - \phi_h\|^2_{L^2(\Omega)}$, i.e. instead of penalizing the distance of our reconstruction to the 0-function, we instead try to minimize the distance to our input function $\phi_h$. As we want our reconstruction $\Phi$ to essentially be a smoother version of $\phi_h$, it stands to reason the 0-level of the functions should be as close as possible. With this modification we make sure our reconstruction has roughly the same isolines as the input function. This modification can be handled analogously as above and we simply get an additional constant $\alpha_0$ at the RHS.

To arrive at the complete formulation for minimizing our functional $J$ we recall the problem formulation in \eqref{PDEMixed} and so the final problem formulation is: For a given $\phi_h \in \mathbb{P}^1_h(\Omega)$, find $(\Phi_1, \Phi_2, \Phi_3, \lambda) \in H^3(\Omega) \times H^2(\Omega) \times H^1(\Omega) \times \mathbb{P}^1_h(\Omega)$ such that
\begin{equation}
\label{CompletePDE}
\begin{split}
\int_\Omega \alpha_0 \Phi_1 \Psi_1 + \alpha_1 \nabla\Phi_1 \cdot \nabla \Psi_1 - \alpha_2 \nabla \Phi_2 \cdot \nabla \Psi_1 + \alpha_3 \nabla \Phi_3 \cdot \nabla \Psi_1 \D x\\
+
\int_\Omega \Phi_2 \Psi_2 + \nabla \Phi_1 \cdot \nabla \Psi_2 \D x
-
\int_{\partial \Omega}  \left(\nu^T \nabla \Phi_1\right) \Psi_2  \D S\\
+
\int_\Omega \Phi_3 \Psi_3 + \nabla \Phi_2 \cdot \nabla \Psi_3 \D x
-
\int_{\partial \Omega}  \left(\nu^T \nabla \Phi_2\right) \Psi_3  \D S\\
 +
 \int_{\partial\Omega} (\nu^T(\alpha_2 H_{\Phi_1} - \alpha_3 H_{ \Phi_2} )  \nabla \Psi_1   \D S\\
 +
\int_{\partial\Omega}  \alpha_3(\nu^T \nabla(H_{\Phi_1})) : H_{\Psi_1}   \D S\\
+
\int_\Omega \lambda \Psi_1 \D x
 +
 \int_{\Omega} (\lambda - \Phi_1)\eta \D x
=
(1+\alpha_0)\int_{\Omega} \phi_h \Psi_1 \D x
\end{split}
\end{equation} for all $(\Psi_1, \Psi_2, \Psi_3, \eta) \in H^3(\Omega)\times H^2(\Omega)\times H^1(\Omega)\times \mathbb{P}^1_h(\Omega)$.

In order to solve the problem numerically, we now have to exchange the Sobolev spaces with the typical FEM ansatz spaces. As these spaces are a subset of $H^1(\Omega)$ it should be noted that we still have to investigate whether our approach needs to be amended with regards to the boundary integrals as higher derivatives are present there and the analytical weak problem requires $H^3(\Omega)$ functions. While we will continue with the $H^1(\Omega)$ FEM spaces, further investigation might be required in a follow-up work.

We have to keep in mind that $\lambda \in \mathbb{P}^1_h(\Omega)$ is given from the analytical problem formulation. We also know $\Phi_1$ must be atleast of polynomial order three and $\Phi_2$ must be of order two, because in our boundary conditions derivatives of third order from $\Phi_1$ and of second order from $\Phi_2$ are present.

In this case we have found a stable choice in $\Phi_1, \Phi_2, \Phi_3 \in \mathbb{P}^3_h(\Omega)$ by heuristical means.

\subsubsection{Complete FEM formulation for the case of $k = 2$}
\label{SectionH2Formulation}
In the case of $H^2$ reconstruction, we will at least state the final problem formulation as:
For a given $\phi_h \in \mathbb{P}^1_h(\Omega)$, find $(\Phi_1, \Phi_2, \lambda) \in H^2(\Omega) \times H^1(\Omega) \times \mathbb{P}^1_h(\Omega)$ such that
\begin{align*}
\int_{\Omega} \alpha_0 \Phi_1 \Psi_1 + \alpha_1 \nabla \Phi_1 \cdot \nabla \Psi_1 - \alpha_2 \nabla \Phi_2 \cdot \nabla \Psi_1 \D x\\
+
\int_\Omega \Phi_2 \Psi_2 + \nabla \Phi_1 \cdot \nabla \Psi_2 \D x
-
\int_{\partial \Omega}  \left(\nu^T \nabla \Phi_1\right) \Psi_2  \D S\\
+\int_{\partial\Omega} \alpha_2 \nu^T H_{\Phi_1}\nabla \Psi_1 \D x\\
+
\int_\Omega \lambda \Psi_1 \D x
+
\int_{\Omega} (\lambda - \Phi_1)\eta \D x
=
(1+\alpha_0)\int_{\Omega} \phi_h \Psi_1 \D x
\end{align*}
  
   for all $(\Psi_1, \Psi_2, \eta) \in  H^2(\Omega)\times H^1(\Omega)\times \mathbb{P}^1_h(\Omega)$.

In this case we have found a stable choice for the numerical ansatz spaces in $\Phi_1, \Phi_2 \in \mathbb{P}^2_h(\Omega)$ as well by heuristical means.

\subsection{Rescaling of problem formulation}

During our investigation we initially found it rather difficult to choose the regularization parameters $\alpha_0, ..., \alpha_3$ correctly, as these regularization parameters have to be adjusted according to the cell size $h$ of the triangulation.

We also encountered a second problem in regards with the cell size: If we look at the weak formulation of the PDE, we have the $0$th to the $3$th derivative of $\Phi$ and $\Psi$ present.  When calculating a matrix entry for the linear system we have to insert basis functions of the $\mathbb{P}^3_h(\Omega)$ space into the weak form. Integral values like $\int_\Omega \alpha_0 \Phi_1\Psi_1   \D x$ will decrease when the step size $h$ decreases, while integral values like $\int_\Omega \alpha_1 \nabla \Phi_1 \cdot \nabla
\Psi_1 \D x$ will increase. If the step size $h$ is smaller or larger than 1 the condition number of the linear system quickly explodes and we can't solve the linear system any more by numerical means.

In order to solve both of these rather practical issues, we rescale our problem formulation: Instead of using the domain $\Omega$ we do all our calculations on a rescaled $\tilde{\Omega}$ which is linked by $\tilde{\Omega} := \left\{r  x \in \mathbb{R}^d: x\in \Omega \right\} $ where $r> 0$ is a scaling parameter. By choosing $r := \frac{1}{h_{max(\Omega)}}$ we ensure that $h_{max}(\tilde{\Omega}) = 1$, which seemed to be a good choice in our experiments. It should be noted however that a different choice for the scaling parameter can be appropriate, e.g. if the mesh is locally refined.

As the next step, we have to transform the right hand size by $\tilde{\phi_h}(\tilde{x}) := \phi_h(\frac{1}{r} \tilde{x}  )$ where $\tilde{x} \in \tilde{\Omega}$.
We then solve the problem \eqref{CompletePDE} on $\tilde{\Omega}$ and receive solutions $\tilde{\Phi}_1, \tilde{\Phi}_2, \tilde{\Phi}_3 \in \mathbb{P}^3_h(\tilde{\Omega})$. Then, taking into account the chain rule for differentiation, we transform back by
\begin{align*}
\Phi_1(x) =& \tilde{\Phi}_1(r x),\\
\Phi_2(x) =& r^2\tilde{\Phi}_2(r x),\\
\Phi_3(x) =& r^4\tilde{\Phi}_3(r x),\\
\end{align*} for all $x \in \Omega$.

\section{Numerical experiments}
\subsection{Preliminaries}
\label{SectionNumPrelim}
Before we begin with the numerical test, we have to state a few preliminaries. All our tests were conducted on a two-dimensional domain $\Omega \in \mathbb{R}^2$ and for simplicity all have been performed on a square or rectangle mesh.

For the tests themselves we use the FEniCS toolbox \cite{LoggEtal2012} or rather our derived inhouse XFEM based toolbox miXFEM \cite{mischa_diss}.

As we saw in the analysis, the reconstruction into $H^2$ suffers from a problematic boundary condition. Therefore almost all of our tests will be performed for the $H^3$ reconstruction.

In the case of the $H^3$ reconstruction there is also a very elegant choice for the regularization parameters $\alpha_0,\hdots, \alpha_3$: We will choose $\alpha:=\alpha_0 = \alpha_3$ and experiment with different for $\alpha$, but set $\alpha_1 = \alpha_2$ very small. This way we want to enfore that the gradient of $\Delta \Phi_1$ is as small as possible and the function $\Phi_1$ is close to $\phi_h$ in a $L^2$-sense, while the values of the Laplacian and gradient of $\Phi_1$ are not punished.
In all $H^3$ calculations we will set $\alpha_1 = \alpha_2 = 10^{-14}$.

In the case of $H^2$ reconstruction we will simply set $\alpha_0 = \alpha
_1 = \alpha_2$ and try different values for those.

When comparing the reconstruction to the exact Laplacian, we will need to project the exact Laplacian onto $\mathbb{P}_h^m(\Omega)$ using the projection operator defined in \ref{OperatorDefinition}.

Then the error calculations will be performed numerically between the projected exact Laplacian of the exact function and our reconstructed Laplacian.

Unless otherwise noted, we will perform an $H^3$ reconstruction.

We will also compare our reconstruction to a sort of weak Laplacian $\mathfrak{f}_h$: Let $\Phi_e \in C^\infty(\Omega)$ be abritrary, then there exists a function $\mathfrak{f}\in C^\infty(\Omega)$ such that 
\begin{align}
\label{4Laplace}
\mathfrak{f} = \Delta \Phi_e.
\end{align}
 
 Now if we multiply \eqref{4Laplace} with an abritrary testfunction $\psi \in C^\infty(\Omega)$, intergrate over $\Omega$ and perform integration by parts, we know that $\mathfrak{f}$ fullfills
 \begin{align*}
\int_\Omega \Delta \Phi_e \psi \D x
=
 \int_\Omega \mathfrak{f} \psi  \D x\\
 \Leftrightarrow
- \int_\Omega \nabla \Phi_e \cdot \nabla \psi \D x
+
\int_{\partial \Omega}(\nu^T \nabla \Phi_e)\psi   \D S
=
 \int_\Omega \mathfrak{f} \psi  \D x.
\end{align*}
We will now use this as a problem formulation where $\Phi_e$ is given and we want to search for an $\mathfrak{f}$ fulfilling this equation for every test function. Therefore, we discretize and try to find an FEM function $\mathfrak{f}_h \in \mathbb{P}^m_h(\Omega)$ such that
\begin{align}
\label{WeakMethodDef}
 \int_\Omega \mathfrak{f}_h \psi_h  \D x 
 =
 - \int_\Omega \nabla \phi_h \cdot \nabla \psi_h \D x
+
\int_{\partial \Omega}(\nu^T \nabla \phi_h)\psi_h   \D S 
\end{align} for all $\psi_h \in \mathbb{P}^m_h(\Omega)$. We will use this weak method with $f_h, \phi_h \in \mathbb{P}_h^m(\Omega)$, $m = 1,2$.

A similar method, albeit with a discontinous Galerkin scheme, has for example been employed in \cite{MARCHANDISE2007949}.

\subsection{Laplacian test case}
First, we consider the simple task of finding the Laplacian of a smooth function $\Phi_e \in C^\infty(\Omega)$, where we only know its piecewise linear projection $\phi_h := A_1 \Phi_e$. Here, we will choose $\Omega := (-11,11)^2$ and the function
\begin{align*}
\Phi_e = \sin\left(\frac{1}{\sqrt{2}} x\right)\sin\left(\frac{1}{\sqrt{2}} y\right).
\end{align*}
One can easily calculate $\Phi_e$ fulfills $-\Delta \Phi_e = \Phi_e$.

\begin{figure}[h!]
\centering
\begin{subfigure}{.5\textwidth}
  \centering
  \includegraphics[width=.9\linewidth]{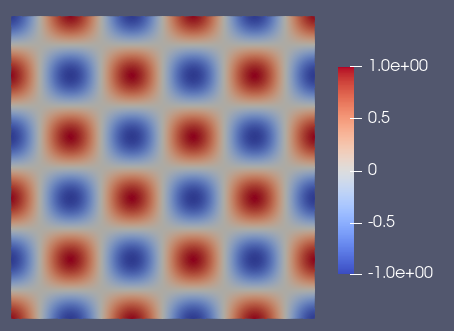}
    \caption{Exact Laplacian}
    \end{subfigure}%
  \hfill
\begin{subfigure}{.5\textwidth}
  \centering
  \includegraphics[width=.9\linewidth]{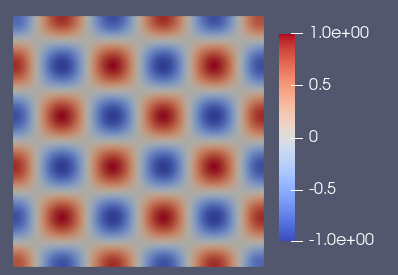}
  \caption{$\alpha = 10$}
\end{subfigure}%

\begin{subfigure}{.5\textwidth}
  \centering
  \includegraphics[width=.9\linewidth]{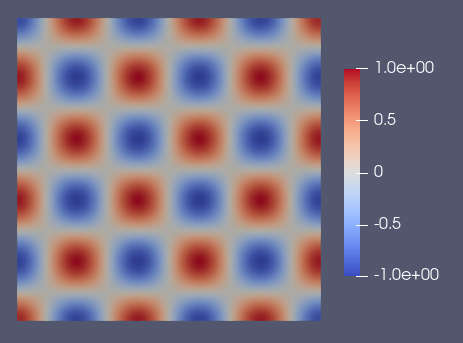}
  \caption{$\alpha = 0.01$}
\end{subfigure}%
\hfill
\begin{subfigure}{.5\textwidth}
  \centering
  \includegraphics[width=.9\linewidth]{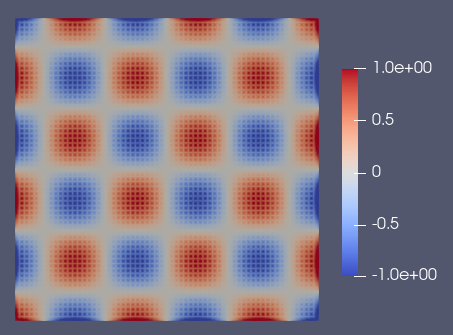}
  \caption{$\alpha = 1e-5$}
\end{subfigure}%

\caption{Tikhonov Reconstruction for Laplacian test case}
\label{fig:TikhLaplace}
\end{figure}

In Figure \ref{fig:TikhLaplace} we can see what effect the values of the different regularization parameter $\alpha$ has. If a large value for $\alpha$ has been chosen, then we get a smoothing or flattening effect towards the boundary, while a very small value results in artifacts. A moderate value smooths out these artifacts while not diminishing the reconstruction properties towards the boundary as hard as a large value for $\alpha$ would.

\begin{figure}[h!]
  \centering
  \includegraphics[width=.7\linewidth]{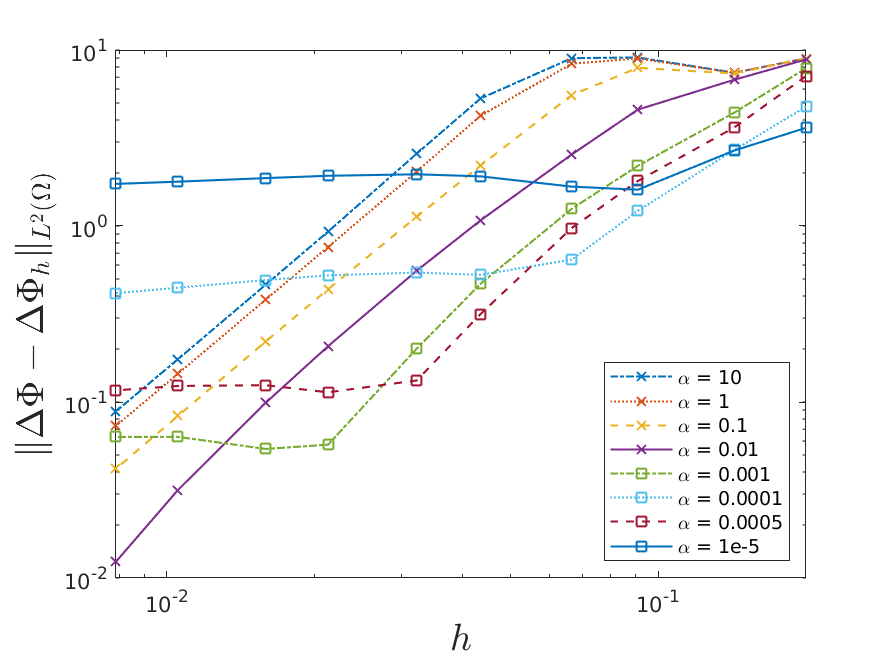}
  \caption{Error graphs for Tikhonov based reconstruction}
  \label{fig:TikhLaplaceConverg}
\end{figure}

Looking at Figure \ref{fig:TikhLaplaceConverg}, we observe how the error term $\|\Phi_2 - A_3(\Delta \Phi_e)\|_{L^2(\Omega)}$ behaves in dependance of the step size $h$ and the regularization parameter $\alpha$. We can make the following observations:
Similar to the visual results in Figure \ref{fig:TikhLaplace}, there seems to be a sweetspot for $\alpha$, where the convergence is optimal. Is $\alpha$ chosen too small the reconstruction artifacts negatively affect convergence. Is $\alpha$ too high the smoothing towards the boundary increases the error. If $\alpha$ is chosen well, we can achieve error-convergence of second polynomial order.

\begin{figure}[h!]
\begin{subfigure}{.5\textwidth}
  \centering
  \includegraphics[width=.9\linewidth]{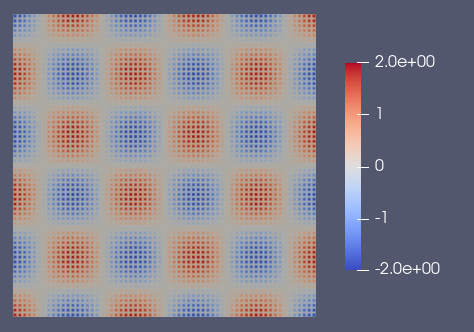}
  \caption{Weak Laplacian}
\end{subfigure}%
\hfill
\begin{subfigure}{.5\textwidth}
  \centering
  \includegraphics[width=.9\linewidth]{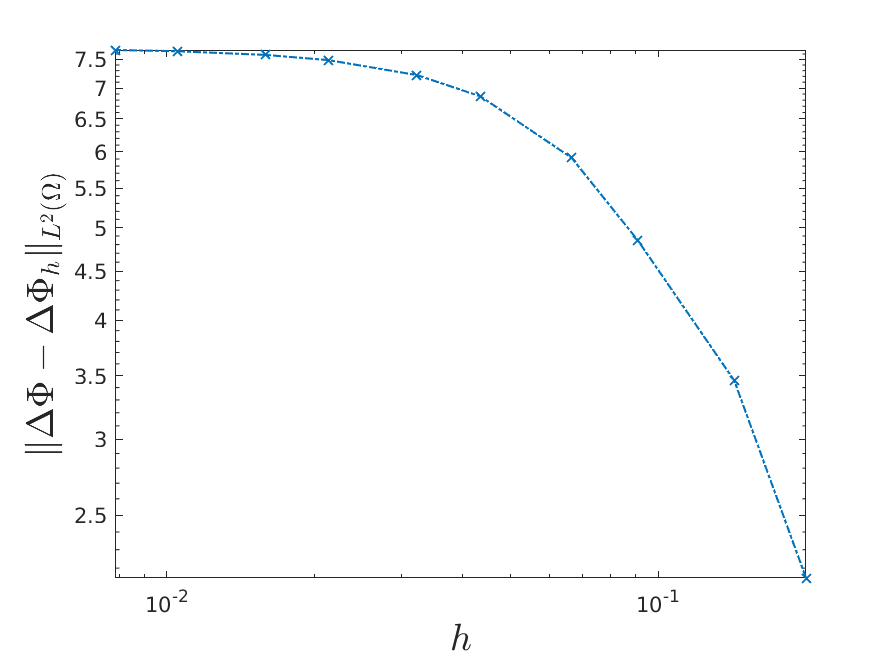}
  \caption{Errors}
\end{subfigure}%

\caption{Weak Laplacian with $\phi_h$ as a piecewise linear function}
\label{fig:WeakLaplaceCG1}
\end{figure}

As we can see in Figure \ref{fig:WeakLaplaceCG1}, the weak method doesn't convergence at all when we restrict ourselves to $\phi_h \in \mathbb{P}^1_h(\Omega)$ and the solution is riddled with artifacts. 

\begin{figure}[h!]
\begin{subfigure}{.5\textwidth}
  \centering
  \includegraphics[width=.9\linewidth]{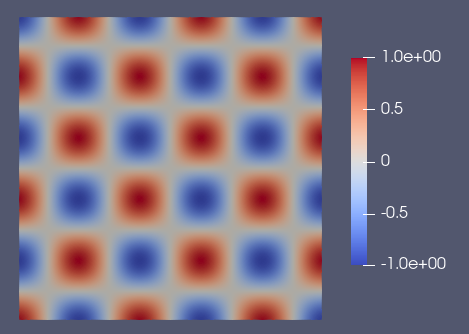}
  \caption{Weak Laplacian}
\end{subfigure}%
\hfill
\begin{subfigure}{.5\textwidth}
  \centering
  \includegraphics[width=.9\linewidth]{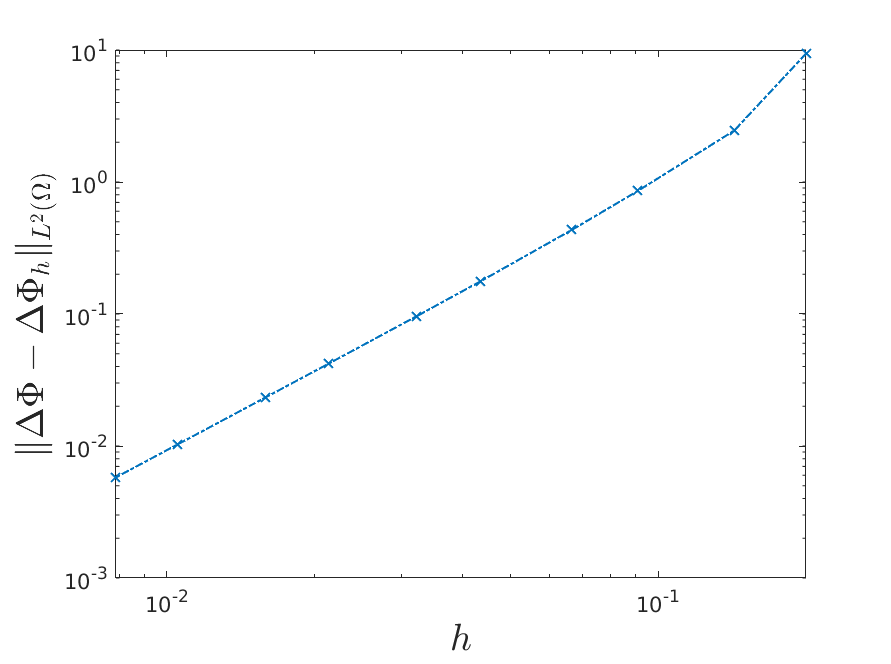}
  \caption{Errors}
  \label{fig:WeakLaplaceCG2Converg}
\end{subfigure}%

\caption{Weak Laplacian with $\phi_h$ as a piecewise quadratic function}
\label{fig:WeakLaplaceCG2}
\end{figure}

It is well understood that we need to increase the polynomial degree of $\phi_h$ and if we try the weak method with $\phi_h \in \mathbb{P}^2_h(\Omega)$, we get convergence of second polynomial order, as we can see in Figure \ref{fig:WeakLaplaceCG2}.

Comparing the error Figures \ref{fig:TikhLaplaceConverg} and \ref{fig:WeakLaplaceCG2Converg} shows us that our Tikhonov based reconstruction  can reconstruct the Laplacian from a piecewise linear function $\phi_h$ almost as good as the weak method from a piecewise quadratic function $\phi_h$, provided we choose the regularization parameter $\alpha$ well. 

So, in this section we have seen we can achieve comparable error convergence for the curvature by using a level set function of one degree lower than in the case of the weak method.

In the next section we will now see another advantage our method has: Level set functions which occur when solving time dependent two-phase Navier-Stokes equations will need to be constructed numerically from a known interface and the process of doing so introduces discretization errors acting similar to noise. A regularization approach has promise to handle such noise better than the weak method we compare with.

\subsection{Static droplet tests}
\label{SectionStaticDroplet}
In this test case, we will not only solve for the Laplacian of a function but also solve a static two-phase Stokes equation. This problem formulation is very similar to the one in \cite{gross2011numerical} and consists of a simple droplet $\Omega_1 := \{x \in \mathbb{R}^2: \|x\| \leq 0.5\}$ inside the square domain $\Omega = (-1,1)$. 

For this we will solve the stationary incompressible Stokes equation
\begin{equation}
\label{StokesProblem}
\begin{aligned}
&- \Div(\mu_i D(u)) + \nabla p = 0 &&\text{~in~} \Omega_i,\\
&\Div(u) = 0 &&\text{~in~} \Omega_i,\\
&\left[\sigma \right]_\Gamma n_\Gamma = - \tau \kappa n_\Gamma,~ [u] = 0  &&\text{~on~} \Gamma,\\
&u = 0 &&\text{~on~} \partial \Omega.
\end{aligned}
\end{equation}

The interface $\Gamma$ will then consist of a circle with radius $0.5$ centered around the 0-point. A corresponding analytical level set function is $\Phi_e(x) := 0.5 - |x|$. 

As there is no outside force and the interface is a circle the exact solution of the Stokes equation is known as
\begin{align*}
&u = 0,
&p = \left\{\begin{array}{ll} c_0 + \kappa, & x\in \Omega_1 \\
         c_0 , & x\in \Omega_2\end{array}\right. 
\end{align*} with a constant $c_0 \in \mathbb{R}$.
 
In terms of discretization, we employ Taylor-Hood-Elements of order $u\in \mathbb{P}^2_h(\Omega_i)$, $p \in \mathbb{P}^1_h(\Omega_i)$ and the weak form we solve is constructed the same way as in \cite{gross2011numerical}.
 
The discrete level set function $\phi_h$ is given over the entire domain $\Omega$ and its zero level describes the discritized interface $\Gamma_h$. We will ensure, through different means, that $\phi_h$ roughly fulfills $|\nabla \phi_h| = 1$ on $\Gamma_h$. 

In this test setting we can of course just project the analytical level set function $\Phi_e(x) = 0.5 - |x|$ onto the spaces $\mathbb{P}^k_h(\Omega)$, $k = 1,2$. One can easily verify it is in fact a signed distance function. But when calculating the time dependent two phase flow problem, the interface $\Gamma_h$ will usually only be explicitly known in the first time step and be part of the solution for the following time steps. 

In the numerical setting, therefore usually only the interface $\Gamma_h$ is known. Then a level set function that is a signed distance function has to be constructed from the interface. The same problem arises for more complicated starting shapes, where it is easy to calculate \textit{a} level set function but not necessarily easy to find one that also fulfills $|\nabla \phi_h| = 1$.

A simple brute force method can find such a corresponding level set function $\phi_h$ by iterating over each vertex $v_i$ of the triangulation  $\mathcal{T}_h$ and setting $\phi_h(x(v_i)) := \mathrm{sign}(x(v_i))\min_{y \in \Gamma_h} | x(v_i) - y |$, where $x(v_i)$ are the coordinates of the vertex $v_i$. The  $\mathrm{sign}(x(v_i))$ term will simply assign a sign to each position $x\in \Omega$, e.g. $ \mathrm{sign}(x) = 1$ for $x \in \Omega_1$ and  $ \mathrm{sign}(x) = -1$ for $x \in \Omega_2$.

We will refer to such a level set function as a 'Numerical Signed Distance Funtion' in the further text. Now accurately calculating the distances for every single vertex $v_i$ is often superfluous and numerically expensive, as we only need the level set function to fulfill $|\nabla \phi_h| = 1$ around the interface $\Gamma_h$.

A modification can lie in only calculating the exact distances to $\Gamma_h$ on a very narrow band of cells around $\Gamma_h$. Then these values are propagated into the rest of the domain with the Fast Marching Method \cite{Sethian1996AFM}. This is now computally much less expensive than the brute force method described above. We will refer to such a level set function as a 'FMM level set function'.

In the following two phase flow tests we will therefore test our method as well as the weak method described in \eqref{WeakMethodDef} on an exactly given level set, a numerical signed distance function and an FMM level set function.

We will evaluate not only the error norms for the curvature but also the $L^2$-error norms for both the velocity $u$ and the pressure $p$ when solving the static Stokes problem described in \eqref{StokesProblem}. 
The error norm $\|\Delta \phi_h - \kappa\|_{L^2(\Gamma_h)}$ will only be evaluated on the interface $\Gamma_h$, since for the Navier-Stokes equation the curvature is only needed in the surface tension functional along $\Gamma_h$.

\clearpage
\paragraph{Exact Interface}

\begin{figure}[h!]
  \centering
\begin{subfigure}{.5\textwidth}
  \centering
  \includegraphics[width=.9\linewidth]{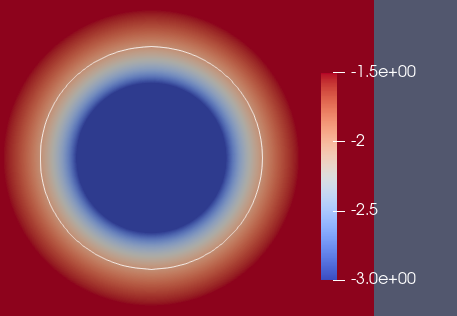}
  \caption{$\alpha = 10$}

\end{subfigure}%
 \hfill
\begin{subfigure}{.5\textwidth}
  \centering
  \includegraphics[width=.9\linewidth]{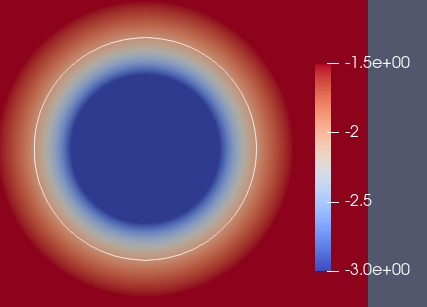}
  \caption{$\alpha = 0.01$}
\end{subfigure}%

\begin{subfigure}{.5\textwidth}
  \centering
  \includegraphics[width=.9\linewidth]{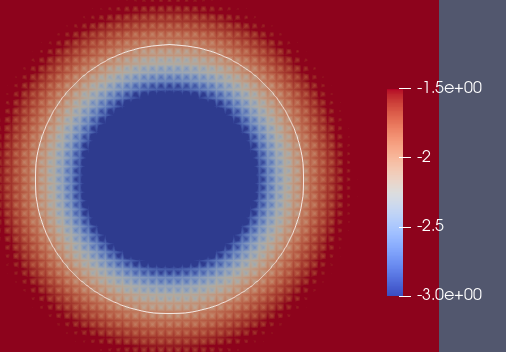}
  \caption{$\alpha = 0.0001$}
\end{subfigure}%
 \hfill
\begin{subfigure}{.5\textwidth}
  \centering
  \includegraphics[width=.9\linewidth]{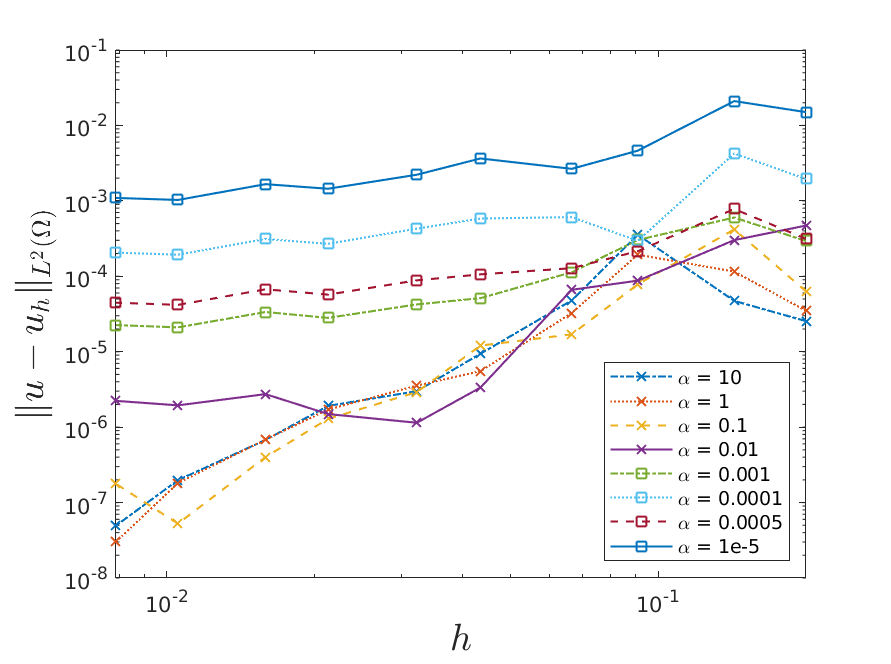}
  \caption{Velocity error}
\end{subfigure}%

\begin{subfigure}{.5\textwidth}
  \centering
  \includegraphics[width=.9\linewidth]{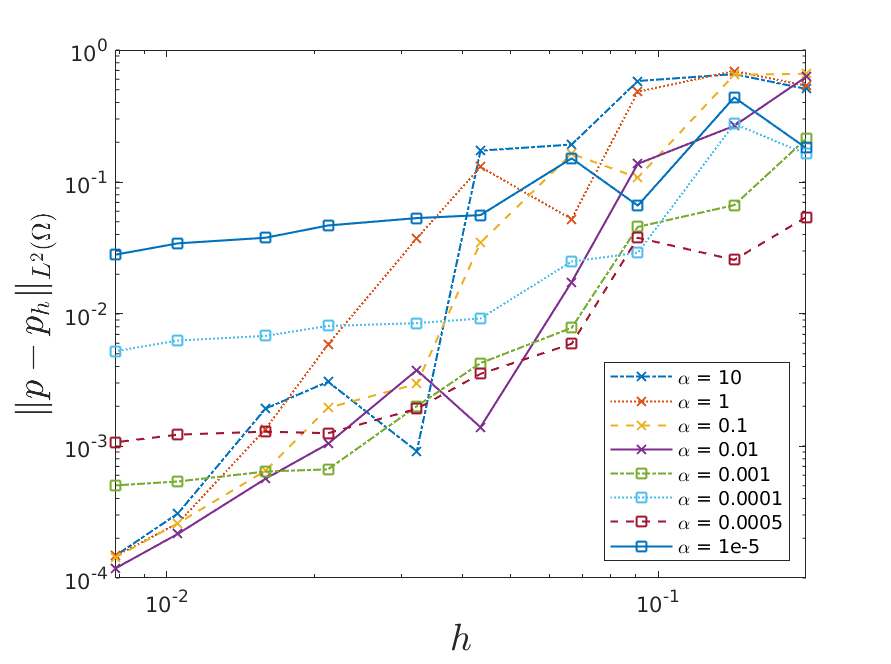}
  \caption{Pressure error}
\end{subfigure}%
 \hfill
\begin{subfigure}{.5\textwidth}
  \centering
  \includegraphics[width=.9\linewidth]{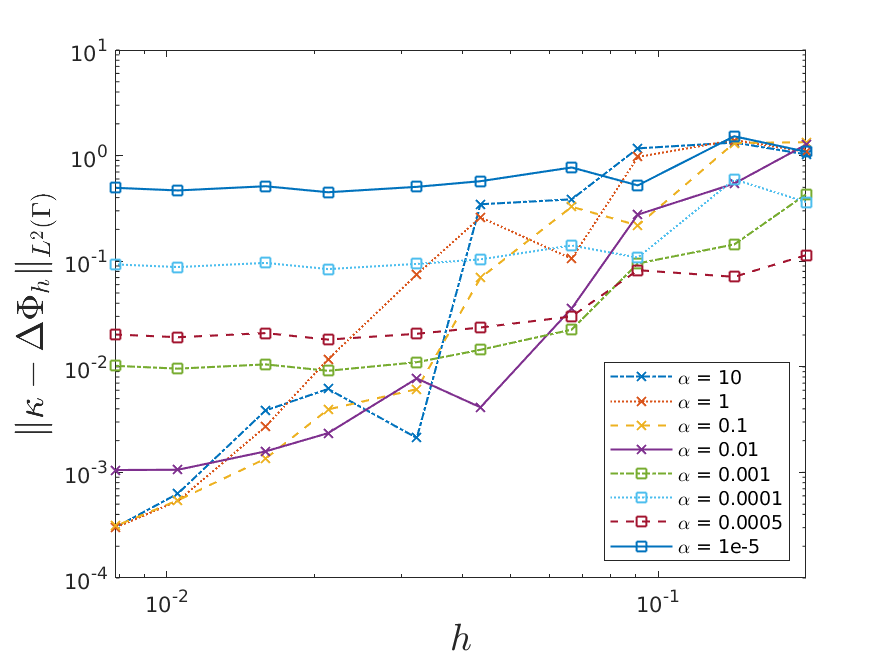}
  \caption{Curvature error}
\end{subfigure}%

\caption{Static droplet results for the exact interface function}
  \label{fig:CenterDropletExact}
\end{figure}
In this first test we will solve the two phase Stokes equation for the static droplet with the projected exact level set function. As we can see on both the resulting images for different $\alpha$ as well as the error graphs in Figure \ref{fig:CenterDropletExact} we are required to set the regularization parameter $\alpha$ above a certain treshold in order to achieve good convergence properties. If $\alpha$ is chosen too small, then we get a lot of artifacts in our reconstruction. Similar to the Laplacian test case, if we compare our Tikhonov based method with the weak method in Figure \ref{fig:WeakDropletExactCG2}, we can see that we can get similar convergence results while reconstructing only from a piecewise linear level set function.

\begin{figure}[h!]
\begin{subfigure}{.5\textwidth}
  \centering
  \includegraphics[width=.9\linewidth]{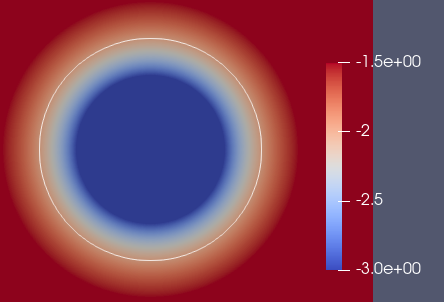}
  \caption{Weak method solution}
\end{subfigure}%
\hfill
\begin{subfigure}{.5\textwidth}
  \centering
  \includegraphics[width=.9\linewidth]{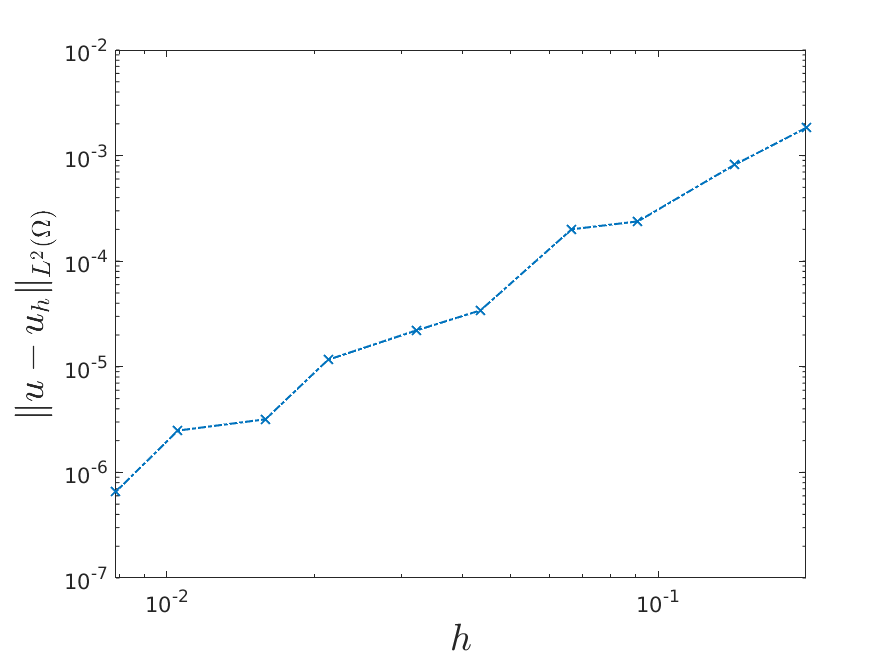}
  \caption{Velocity error}
  \end{subfigure}%
  
  \begin{subfigure}{.5\textwidth}
  \centering
  \includegraphics[width=.9\linewidth]{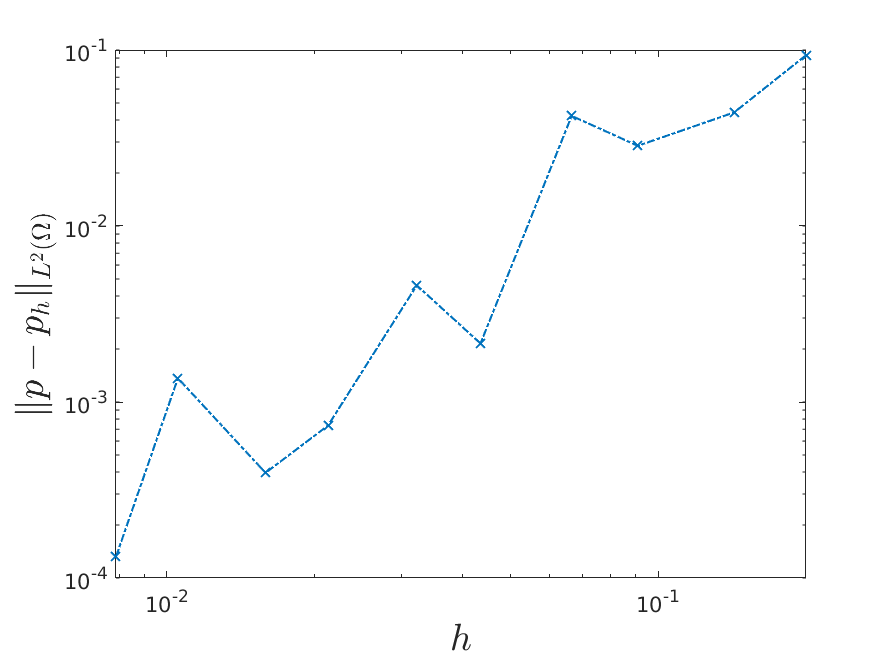}
  \caption{Pressure error}
\end{subfigure}%
\hfill
\begin{subfigure}{.5\textwidth}
  \centering
  \includegraphics[width=.9\linewidth]{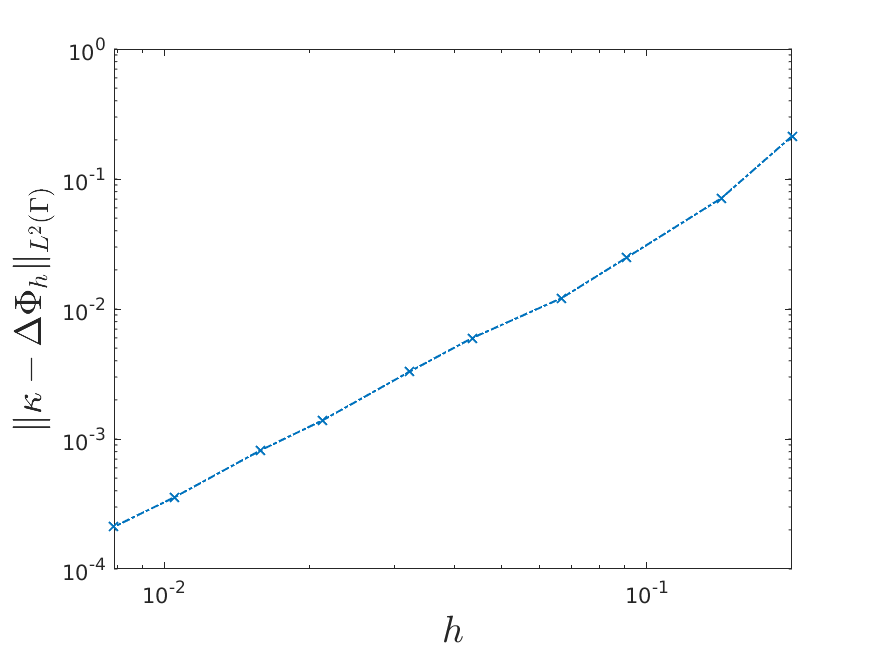}
  \caption{Curvature error}

\end{subfigure}%

\caption{Static droplet weak curvature results with $\phi_h$ as a piecewise quadratic function}
\label{fig:WeakDropletExactCG2}
\end{figure}

\paragraph{Influence of mesh geometry}

During our investigation we found that the geometry of the underlying mesh can have a huge influence on the quality of the Tikhonov reconstruction. In Figures \ref{fig:CrossedGeometry} and \ref{fig:RightGeometry} we can see how the domain $\Omega = (-1,1)^2$ is triangulated by different meshes. We have dubbed the two structured meshes  'Crossed mesh' and 'Diagonal mesh' corresponding to the direction of the diagonals.

\begin{figure}[h!]
  \centering
\begin{subfigure}{.5\textwidth}
  \centering
  \includegraphics[width=.9\linewidth]{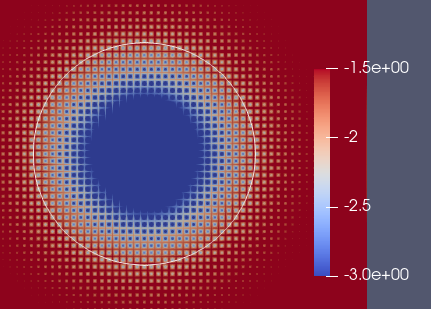}
  \caption{Crossed mesh solution with $\alpha = 1e-5$}
\end{subfigure}%
 \hfill
 \begin{subfigure}{.5\textwidth}
  \centering
  \includegraphics[width=.9\linewidth]{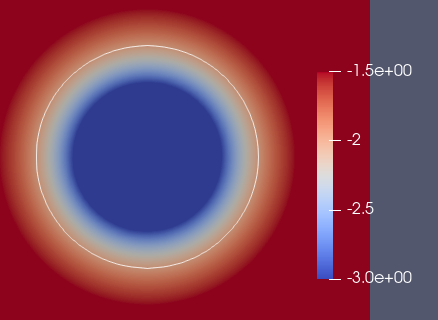}
  \caption{Right mesh solution with $\alpha = 1e-5$}
\end{subfigure}%

\begin{subfigure}{.5\textwidth}
  \centering
  \includegraphics[width=.9\linewidth]{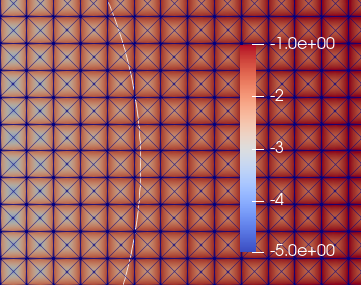}
  \caption{Crossed mesh geometry}
  \label{fig:CrossedGeometry}
\end{subfigure}%
 \hfill
 \begin{subfigure}{.5\textwidth}
  \centering
  \includegraphics[width=.9\linewidth]{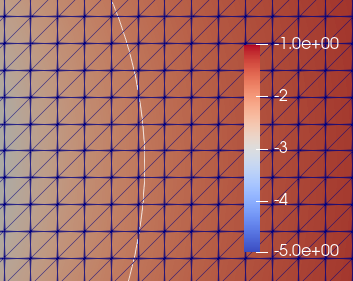}
  \caption{Diagonal mesh geometry}
    \label{fig:RightGeometry}
\end{subfigure}%

 \begin{subfigure}{.5\textwidth}
  \centering
  \includegraphics[width=.9\linewidth]{CenterDroplet/exact/errorK.png}
  \caption{Curvature error for Crossed geometry}
\end{subfigure}%
 \hfill
 \begin{subfigure}{.5\textwidth}
  \centering
  \includegraphics[width=.9\linewidth]{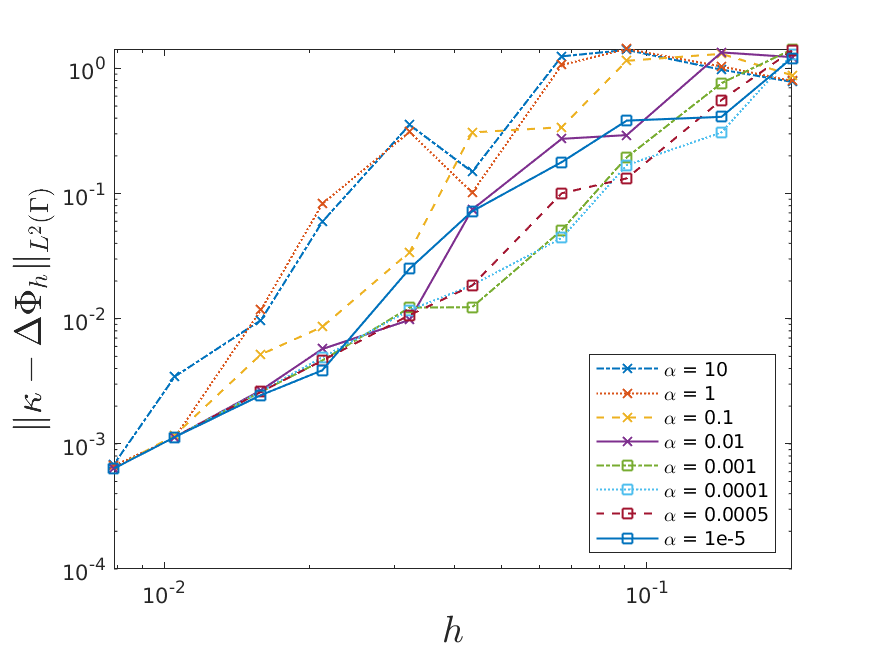}
  \caption{Curvature error for Diagonal geometry}
\end{subfigure}%

\caption{Influence of the mesh geometry on the errors}
\label{fig:Geometry}
\end{figure}

As we can see in the error graphs for the curvature in Figure \ref{fig:Geometry}, the reconstruction converges significantly better on the diagonal mesh compared to the crossed mesh for very low regularization parameters $\alpha$ and the diagonal mesh experiences no artifacts in these cases, unlike the crossed mesh.

Apart from this small demonstration of how the mesh geometry influences the reconstruction quality, we will always calculate on the crossed mesh.

\clearpage
\paragraph{Numerical signed distance function}

In this test case we will now see how the numerical signed distance function performs in terms of the reconstruction quality. The interface $\Gamma_h$ still prescribes a discretized circle with radius $0.5$ centered around $0$ and in fact the interface $\Gamma_h$ is the same as in the projected exact case. While the exactly given interface earlier is a purely academical application for calculating the interface curvature, the numerical signed distance function provides an example mucher closer to a real world application.

Let us start with the weak method in this test case.

\begin{figure}[h!]
\begin{subfigure}{.5\textwidth}
  \centering
  \includegraphics[width=.9\linewidth]{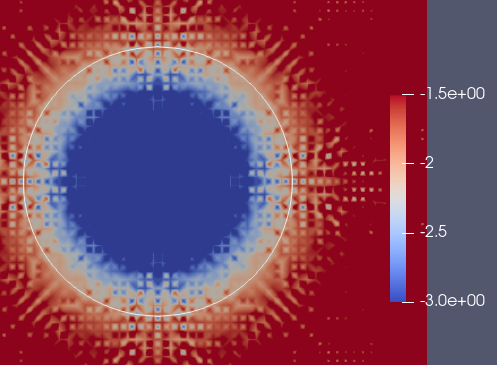}
  \caption{Weak method solution}
\end{subfigure}%
\hfill
\begin{subfigure}{.5\textwidth}
  \centering
  \includegraphics[width=.9\linewidth]{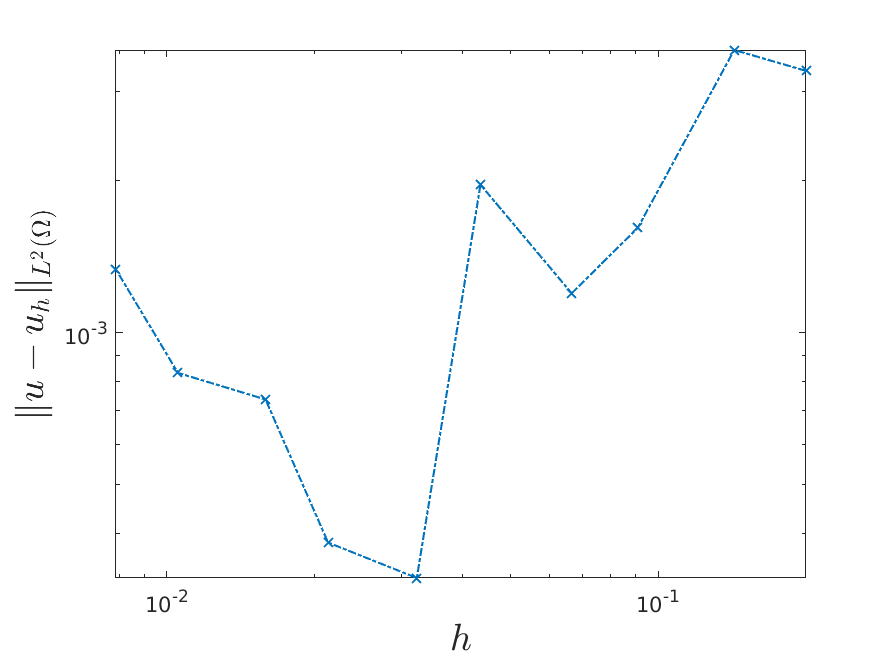}
  \caption{Velocity error}
  \end{subfigure}%
  
  \begin{subfigure}{.5\textwidth}
  \centering
  \includegraphics[width=.9\linewidth]{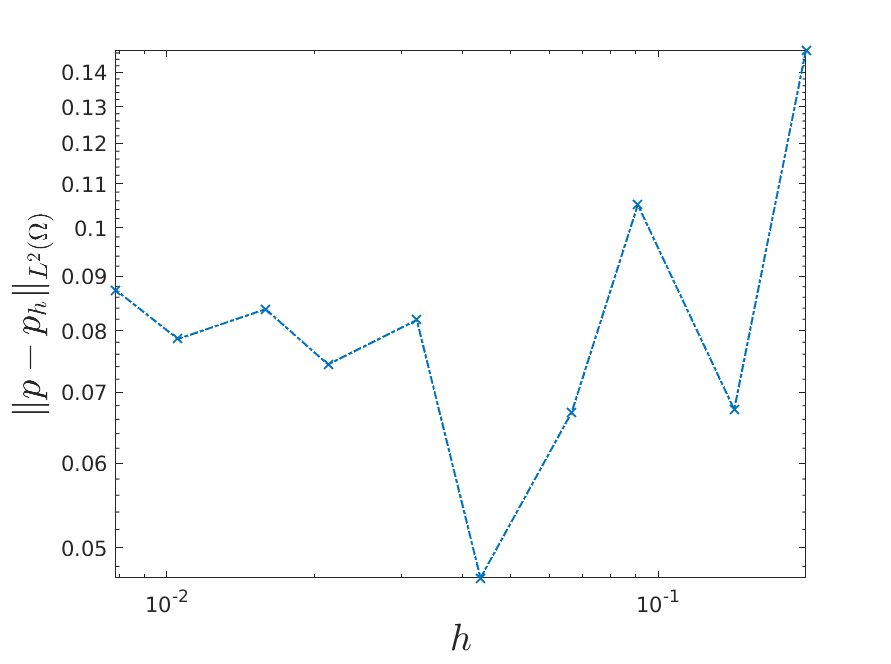}
  \caption{Pressure error}
\end{subfigure}%
\hfill
\begin{subfigure}{.5\textwidth}
  \centering
  \includegraphics[width=.9\linewidth]{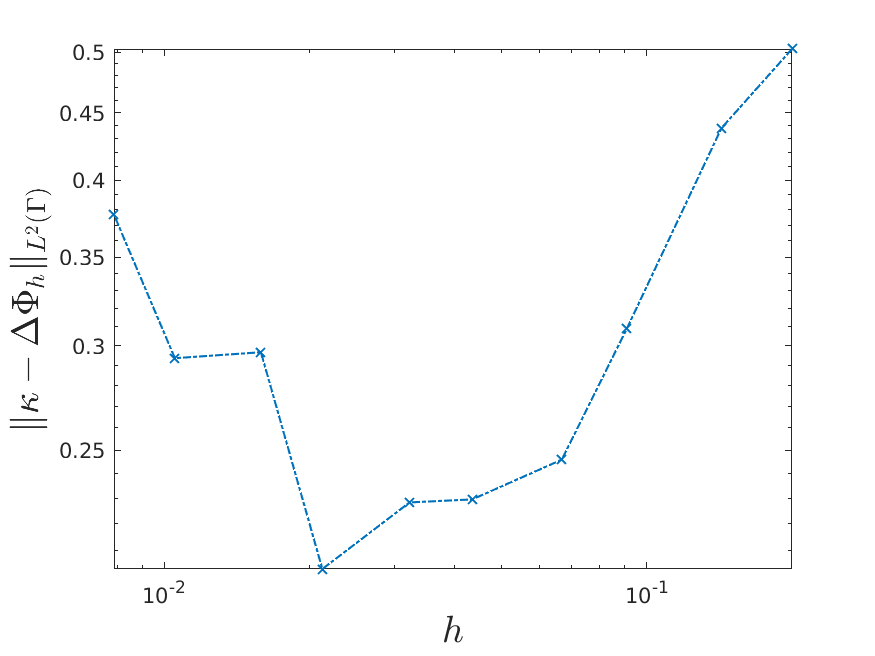}
  \caption{Curvature error}

\end{subfigure}%

\caption{Static droplet weak curvature results with $\phi_h$ as a piecewise quadratic function for the numerical signed distance function}
\label{fig:WeakDropletAllSignedCG2}
\end{figure}

As we can see in Figure \ref{fig:WeakDropletAllSignedCG2}, the weak method's convergence is practically non-existent in this case and the curvature error stays constant or even gets worse if the step size $h$ decreases. In the same figure we can also see that the weak's method solution has significant artifacts unlike in the exact interface case. The reason for this behavior lies in how the numerical signed distance function is constructed. Since we try to minimize the distances to the discretized interface $\Gamma_h$, we get small disturbences in the values compared to the exact signed distance function $\Phi_e$ which minimizes the distances to the exact interface $\Gamma$. It is natural to think of these disturbences as noise added onto the exact level set function. Differentiation is a very ill-posed problem in general and therefore small noise on the data $\phi_h$ will get amplified tremendously when trying to differentiate.
Now let us compare these results to our Tikhonov based $H^3$-reconstruction.

\begin{figure}[h!]
  \centering
\begin{subfigure}{.5\textwidth}
  \centering
  \includegraphics[width=.9\linewidth]{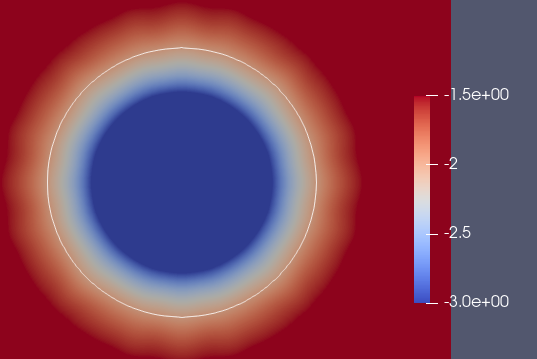}
  \caption{$\alpha = 10$}
\end{subfigure}%
 \hfill
\begin{subfigure}{.5\textwidth}
  \centering
  \includegraphics[width=.9\linewidth]{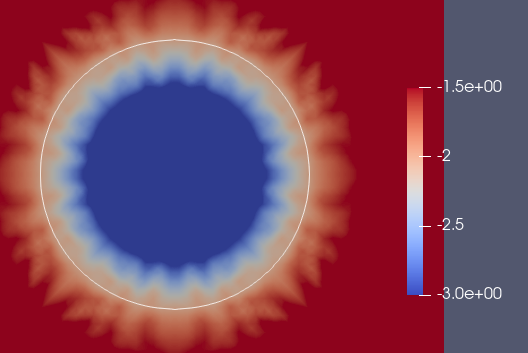}
  \caption{$\alpha = 0.01$}
\end{subfigure}%

\begin{subfigure}{.5\textwidth}
  \centering
  \includegraphics[width=.9\linewidth]{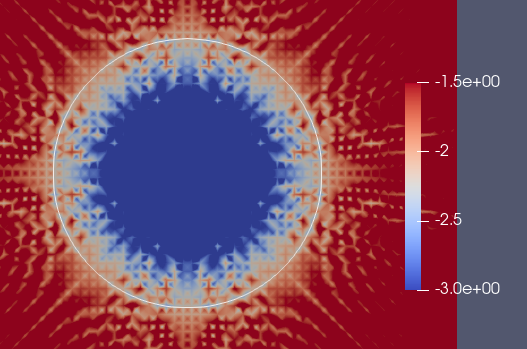}
  \caption{$\alpha = 0.0001$}
\end{subfigure}%
 \hfill
\begin{subfigure}{.5\textwidth}
  \centering
  \includegraphics[width=.9\linewidth]{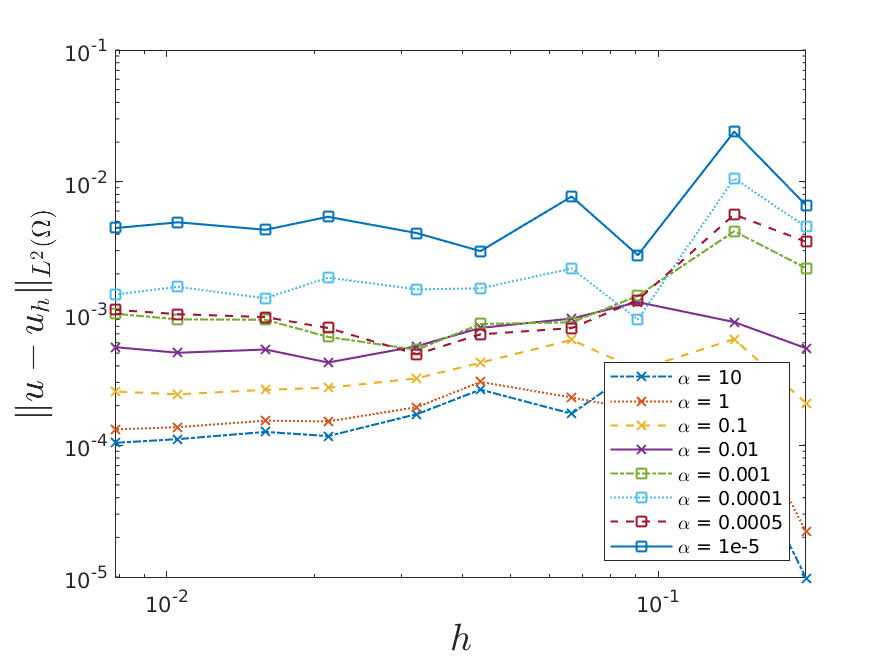}
  \caption{Velocity error}
\end{subfigure}%

\begin{subfigure}{.5\textwidth}
  \centering
  \includegraphics[width=.9\linewidth]{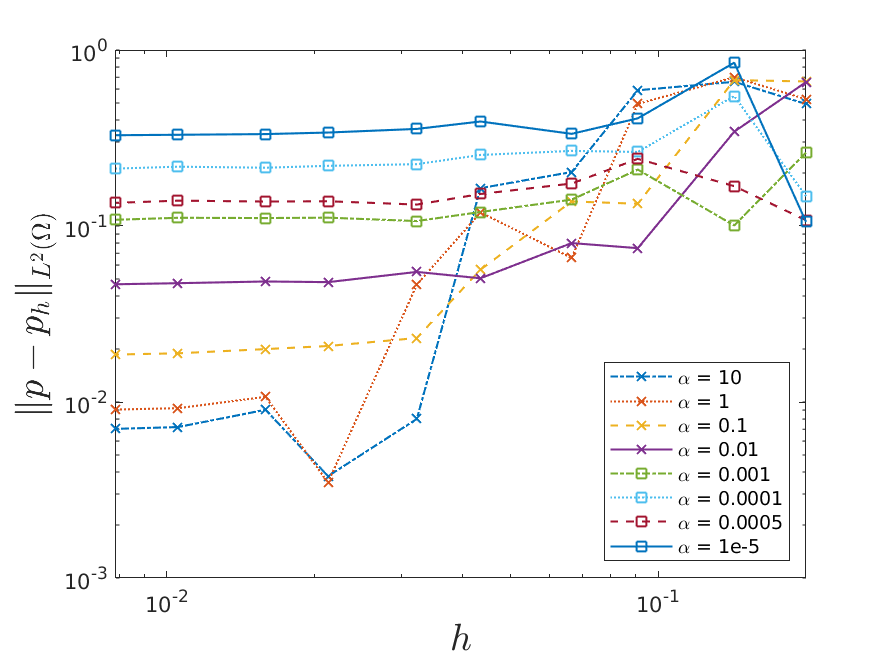}
  \caption{Pressure error}
\end{subfigure}%
 \hfill
\begin{subfigure}{.5\textwidth}
  \centering
  \includegraphics[width=.9\linewidth]{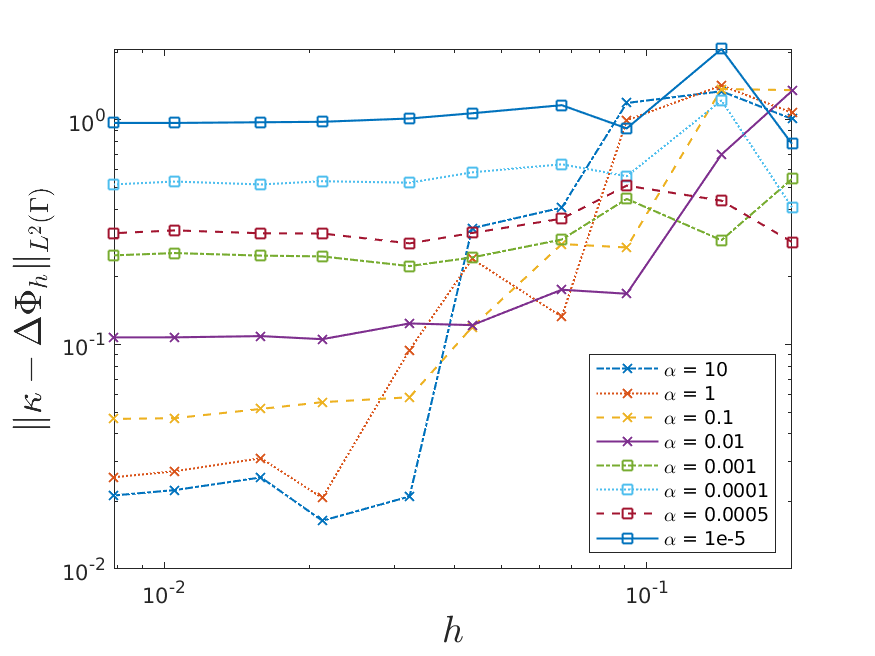}
  \caption{Curvature error}
\end{subfigure}%

\caption{Static droplet results for the numerical signed interface function}
  \label{fig:CenterDropletAllSigned}
\end{figure}

And as we can see in Figure \ref{fig:CenterDropletAllSigned}, we can achieve a much smoother Laplacian for $\phi_h$ piecewise linear compared to the weak method. Unfortunately, we do not get strict convergence for the errors. Instead, we see that the errors (for high $\alpha$) will sharply fall off and then stay constant or just be almost constant from the beginning for a low $\alpha$. Now this behavior can be be explained by how a higher regularization parameter smooths out the noise whereas for a lower $\alpha$, the noise is not regularized away and therefore increases the error term.

\clearpage
\paragraph{Comparison between $H^2$ and $H^3$ reconstruction}

As we mentioned earlier, almost all of our tests were performed with the $H^3$-reconstruction, as defined in Section \ref{SectionH3Formulation}. In this paragraph, we will investigate for the above test case (static droplet decribed with a numerical signed distance function) how the $H^2$-reconstruction, as defined in Section \ref{SectionH2Formulation}, performs.

\begin{figure}[h!]
  \centering
\begin{subfigure}{.5\textwidth}
  \centering
  \includegraphics[width=.9\linewidth]{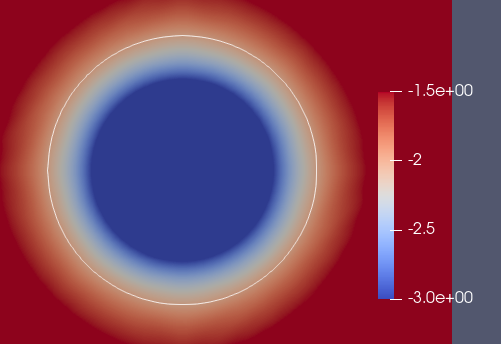}
  \caption{$\alpha = 10$}
\end{subfigure}%
 \hfill
\begin{subfigure}{.5\textwidth}
  \centering
  \includegraphics[width=.9\linewidth]{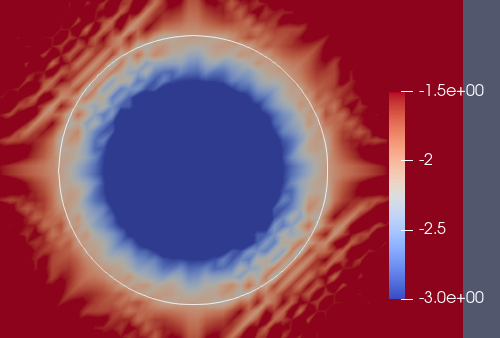}
  \caption{$\alpha = 0.01$}
\end{subfigure}%

\begin{subfigure}{.5\textwidth}
  \centering
  \includegraphics[width=.9\linewidth]{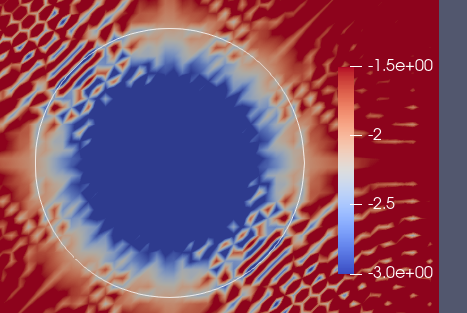}
  \caption{$\alpha = 0.001$}
\end{subfigure}%
 \hfill
\begin{subfigure}{.5\textwidth}
  \centering
  \includegraphics[width=.9\linewidth]{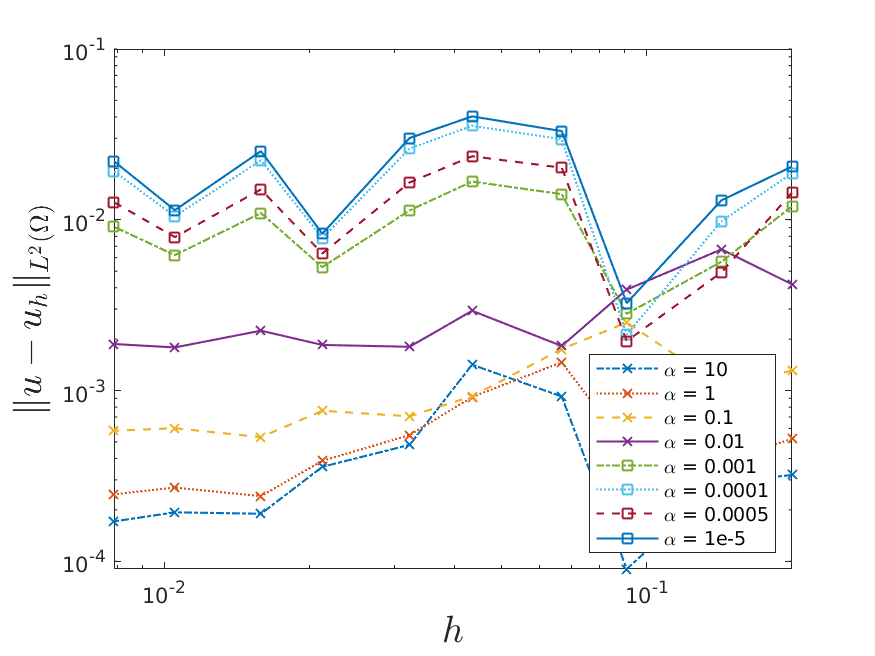}
  \caption{Velocity error}
\end{subfigure}%

\begin{subfigure}{.5\textwidth}
  \centering
  \includegraphics[width=.9\linewidth]{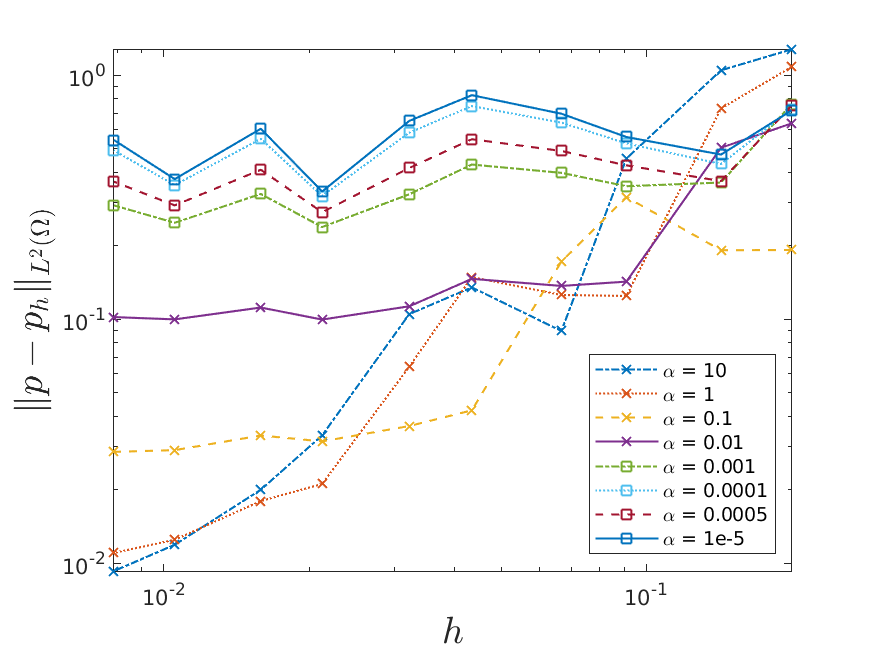}
  \caption{Pressure error}
\end{subfigure}%
 \hfill
\begin{subfigure}{.5\textwidth}
  \centering
  \includegraphics[width=.9\linewidth]{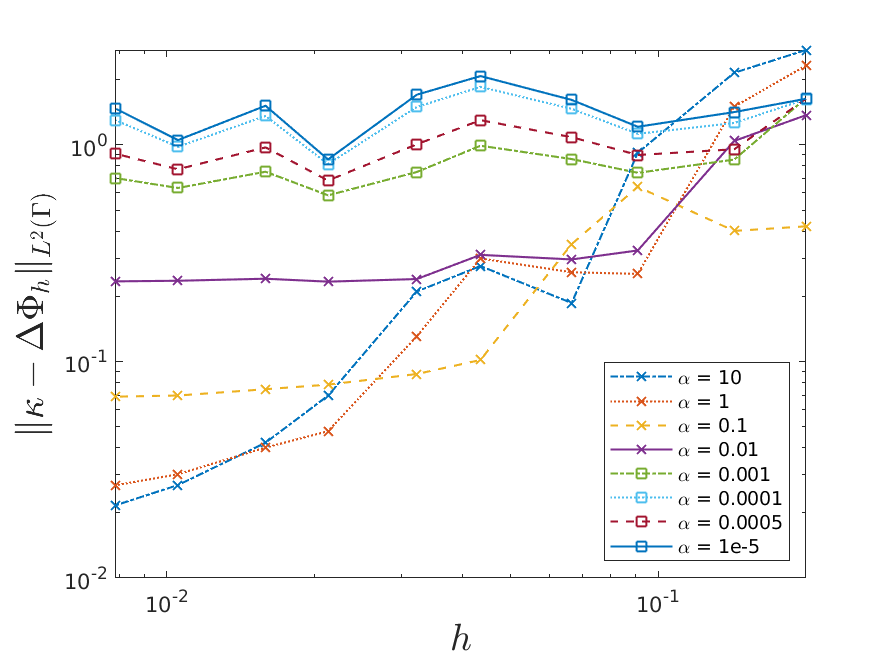}
  \caption{Curvature error}
\end{subfigure}%

\caption{Static dropletresults for the numerical signed interface function with $H^2$ reconstruction}
  \label{fig:CenterDropletH2}
\end{figure}

When we compare the $H^3$ results in Figure \ref{fig:CenterDropletAllSigned} against the $H^2$ results in Figure \ref{fig:CenterDropletH2}, we do not see any significant differences for a very high regularization parameter $\alpha$. For very small $\alpha$, the $H^2$ reconstruction seems to perfom slightly better but still bad, whereas for an intermediate $\alpha$ like $\alpha = 0.1$ or $\alpha = 0.01$ the $H^3$-reconstruction seems to be  favored a bit. But it should be mentioned, that the different way the penalty norms are weighted with $\alpha$, as described in Section \ref{SectionNumPrelim}, does not actually permit a direct comparison for a given $\alpha$.

But over all we do not conclude that either the $H^3$ or the $H^2$  reconstruction works better than the other \textit{in the interior of the domain}.

\begin{figure}[h!]
\begin{subfigure}{.5\textwidth}
  \centering
  \includegraphics[width=.9\linewidth]{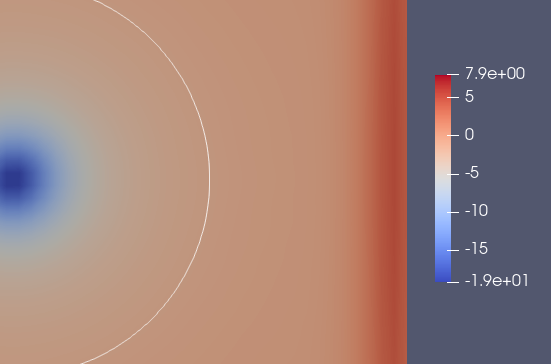}
  \caption{$H^2$ reconstruction with $\alpha = 10$}
\end{subfigure}%
 \hfill
\begin{subfigure}{.5\textwidth}
  \centering
  \includegraphics[width=.9\linewidth]{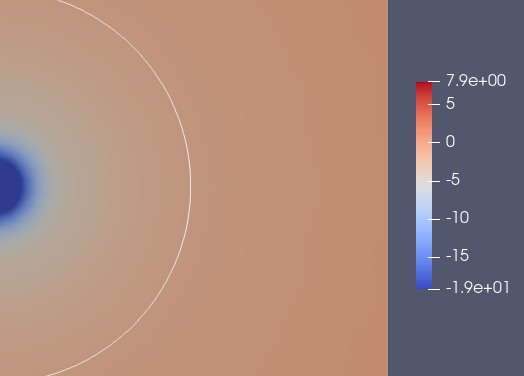}
  \caption{$H^3$ reconstruction with $\alpha = 10$}
\end{subfigure}%
\caption{Behaviour of $H^2$ and $H^3$ reconstruction towards the boundary}
  \label{fig:CenterDropletH2Boundary}
\end{figure}

Now the choice of the color bar in the Figures \ref{fig:CenterDropletExact} and  \ref{fig:CenterDropletH2} did not permit to state anything about the behavior of the solution towards the boundary. We mentioned earlier in Section \ref{SectionChoiceSobolev} that the $H^2$-reconstruction suffers from a problematic boundary condition for the second derivates of $\Phi$. As we can see in Figure \ref{fig:CenterDropletH2Boundary}, where we rescaled the color bar to the highest possible value of the $H^2$-reconstruction, the $H^2$-reconstruction will show a completely different behavior towards the boundary and has a completely wrong value there. In the wetting tests later, we will show the effect of the boundary behavior on the errors for the curvature, velocity and pressure.

\paragraph{FMM level set function}

Now we will present the results for the FMM level set function. While earlier convergence results on the exact level set function are promising, they are purely of academic interest. The results we presented for the numerical signed distance function and the results we are to show here for the FMM level set function are much closer to a real life application. 

\begin{figure}[h!]
  \centering
\begin{subfigure}{.5\textwidth}
  \centering
  \includegraphics[width=.9\linewidth]{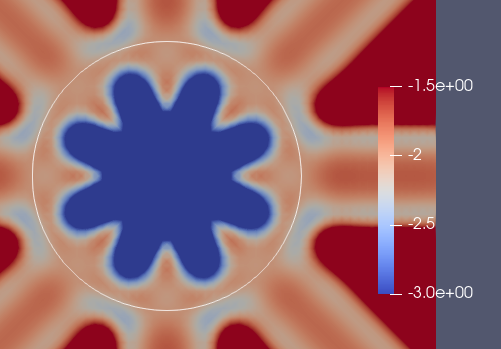}
  \caption{$\alpha = 10$}

\end{subfigure}%
 \hfill
\begin{subfigure}{.5\textwidth}
  \centering
  \includegraphics[width=.9\linewidth]{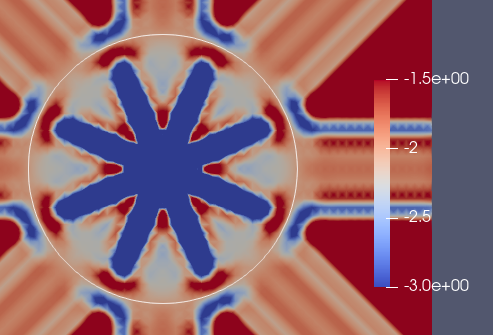}
  \caption{$\alpha = 0.01$}
\end{subfigure}%

\begin{subfigure}{.5\textwidth}
  \centering
  \includegraphics[width=.9\linewidth]{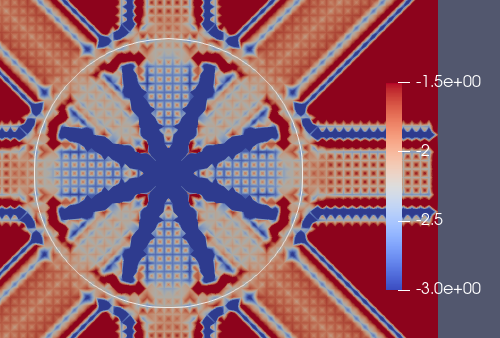}
  \caption{$\alpha = 0.0001$}
\end{subfigure}%
 \hfill
\begin{subfigure}{.5\textwidth}
  \centering
  \includegraphics[width=.9\linewidth]{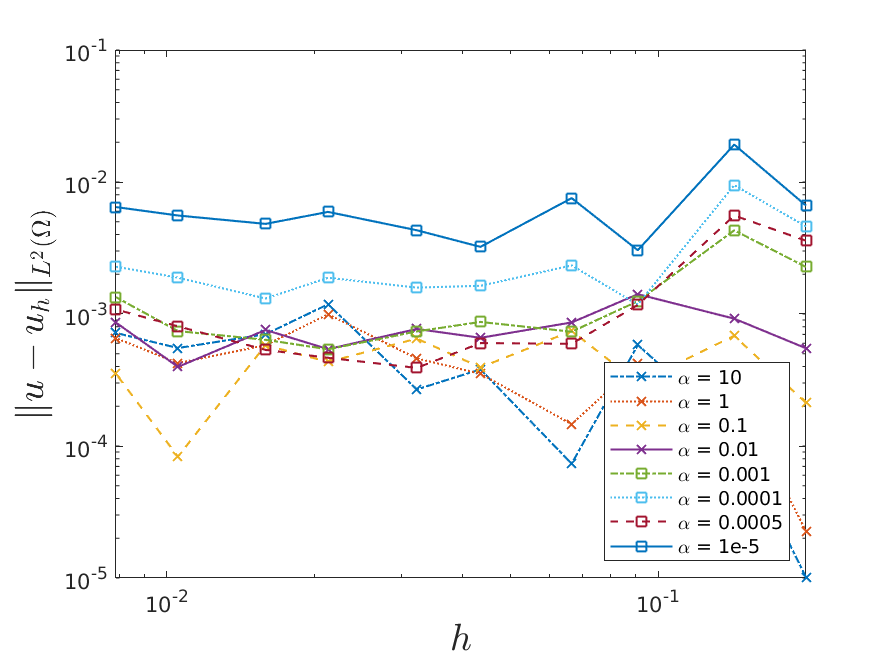}
  \caption{Velocity error}
\end{subfigure}%

\begin{subfigure}{.5\textwidth}
  \centering
  \includegraphics[width=.9\linewidth]{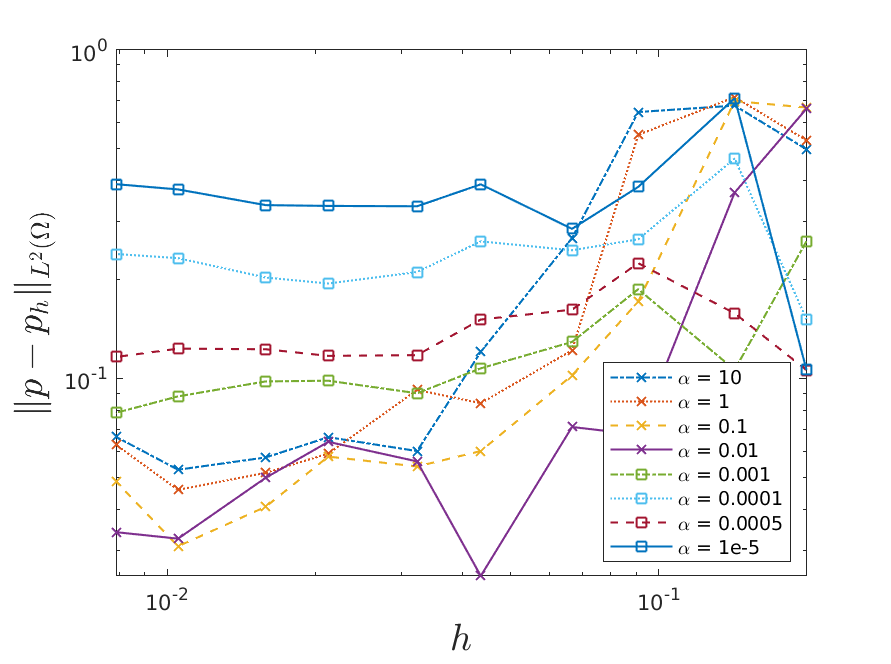}
  \caption{Pressure error}
\end{subfigure}%
 \hfill
\begin{subfigure}{.5\textwidth}
  \centering
  \includegraphics[width=.9\linewidth]{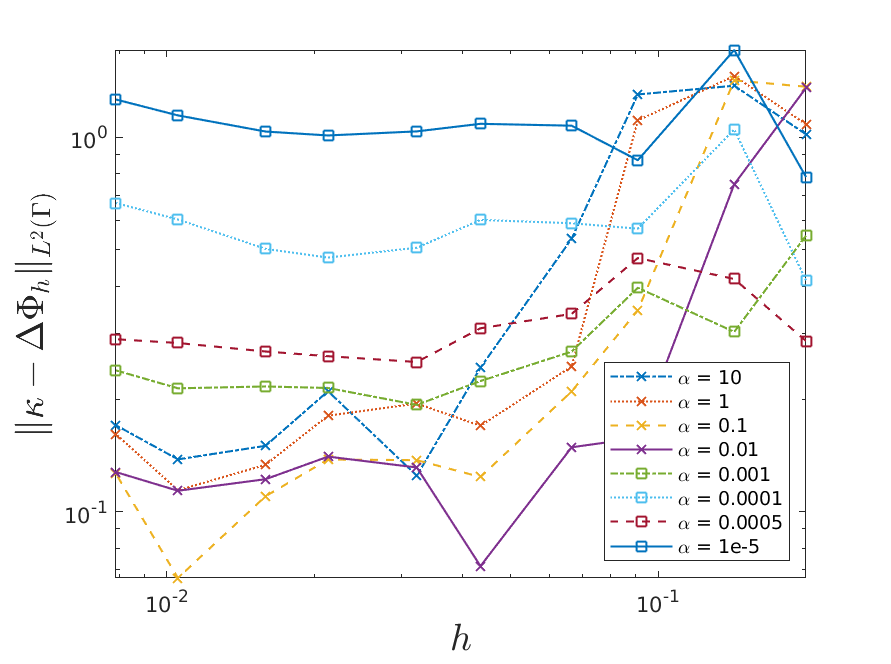}
  \caption{Curvature error}
\end{subfigure}%
\caption{Static droplet results for the FMM interface function}
  \label{fig:CenterDropletFMM}
\end{figure}

As we can see in Figure \ref{fig:CenterDropletFMM}, the error behavior is even worse now than for the numerical signed distance function in Figure \ref{fig:CenterDropletAllSigned}. This should come as no surprise, as the FMM level set function fulfills the signed distance property only on cells surrounding the interface $\Gamma_h$.

\begin{figure}[h!]
  \centering
    \includegraphics[width=.7\linewidth]{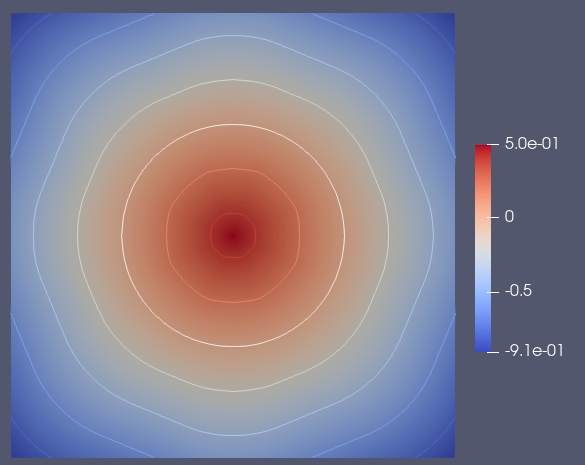}
      \caption{Isolines for FMM level set function}
      \label{fig:FMMIsolines}
\end{figure}

The appearance of the solution is a direct results of the FMM algorithm. As we can see in Figure \ref{fig:FMMIsolines}, the isolines of the FMM level set function do not form concentric circles like for the exact level set function. Instead the isolines further away from the zero level have the appearance of smoothed polygons. As such the curvature suddenly jumps between the flat parts and the smoothed corners of the level set function.

\begin{figure}[h!]
\begin{subfigure}{.5\textwidth}
  \centering
  \includegraphics[width=.9\linewidth]{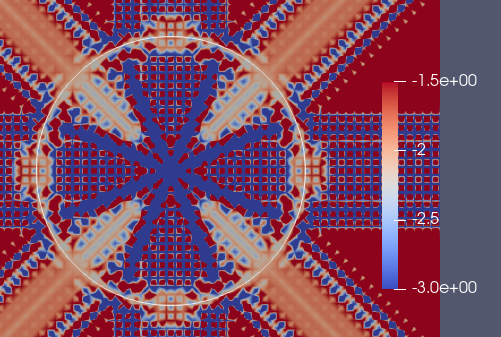}
  \caption{Weak method solution}
\end{subfigure}%
\hfill
\begin{subfigure}{.5\textwidth}
  \centering
  \includegraphics[width=.9\linewidth]{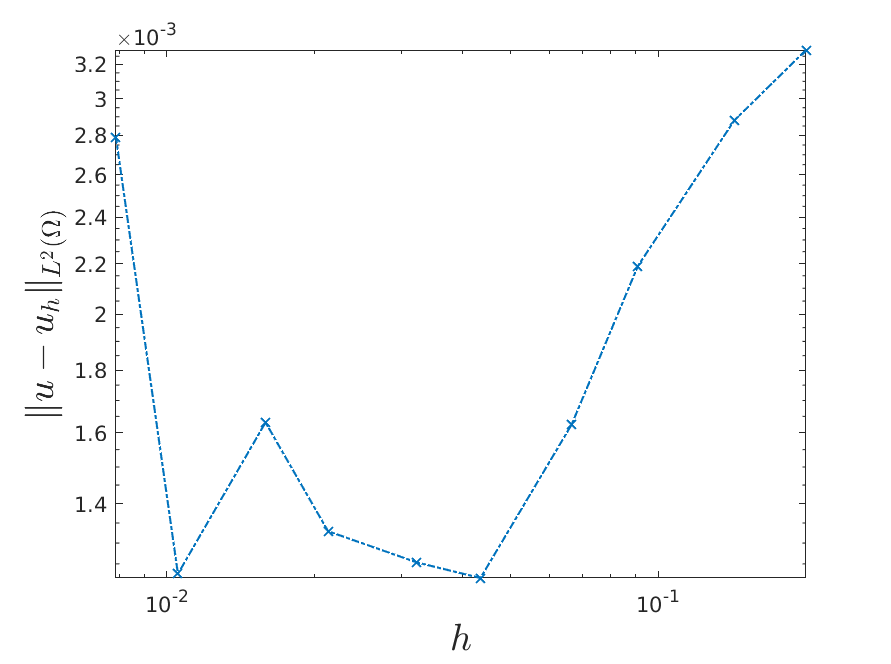}
  \caption{Velocity error}
  \end{subfigure}%
  
  \begin{subfigure}{.5\textwidth}
  \centering
  \includegraphics[width=.9\linewidth]{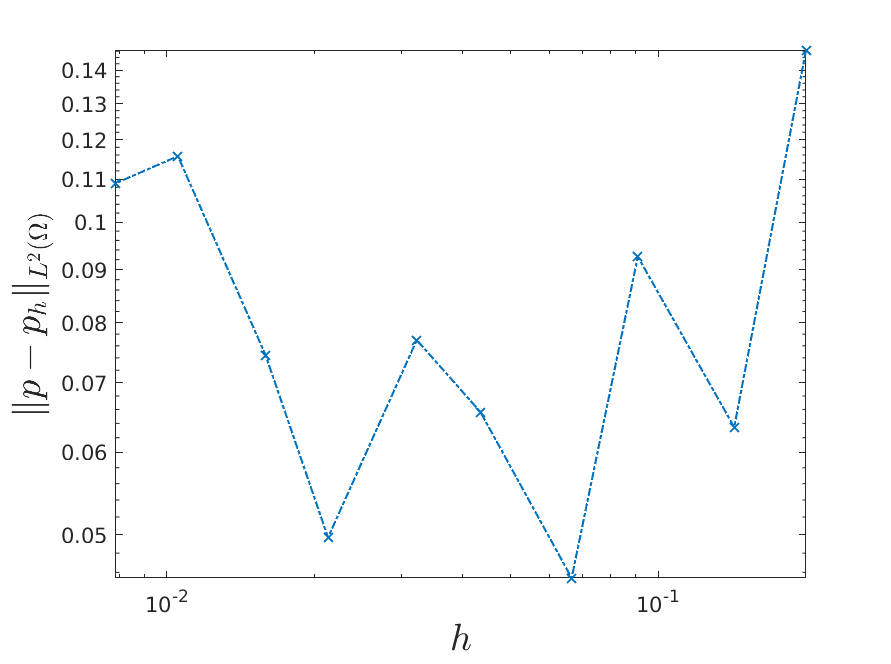}
  \caption{Pressure error}
\end{subfigure}%
\hfill
\begin{subfigure}{.5\textwidth}
  \centering
  \includegraphics[width=.9\linewidth]{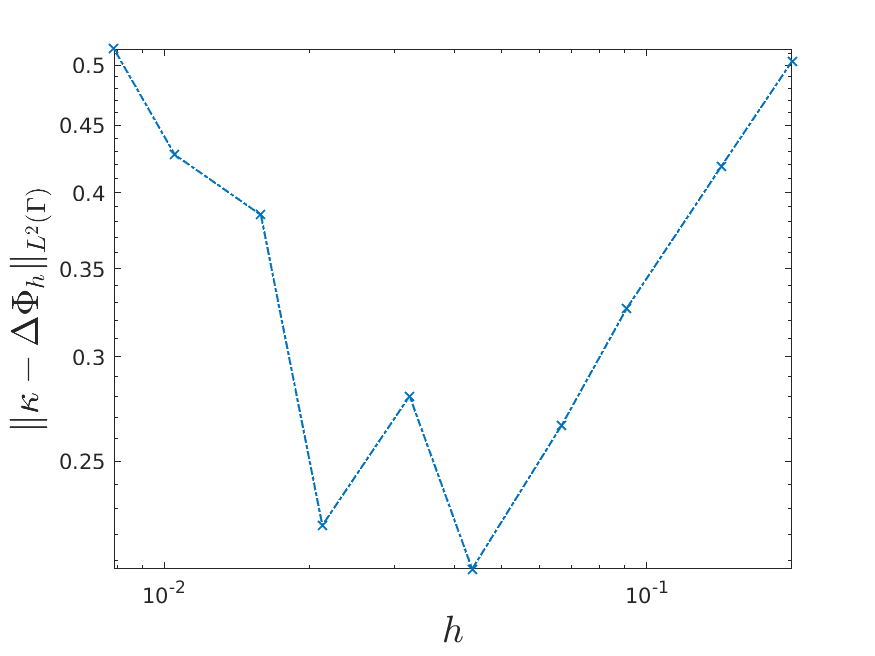}
  \caption{Curvature error}

\end{subfigure}%

\caption{Static droplet weak curvature results with $\phi_h$ as a piecewise quadratic function}
\label{fig:WeakDropletFMMCG2}
\end{figure}

And as we can see in Figure \ref{fig:WeakDropletFMMCG2} the weak method performs very poorly as well on such a level set function. The results are slightly worse as well than in the numerical signed distance case.

\clearpage
\subsection{Wetting tests}

The wettings tests prescribed in this sub-section will only differ in two small differences to the static droplet tests in Section \ref{SectionStaticDroplet}:

First, we will prescribe the exact level set function $\Phi_e$ as 
\begin{align*}
\Phi_e(x) := 0.5 - |x - c|,
\end{align*} with $c \in \mathbb{R}^2$, i.e. our interface $\Gamma$ is a circle with radius $0.5$ with center $c$. We will choose $c$ such that $\Gamma$ intersects $\partial \Omega$. 

The second change lies in the boundary condition for the velocity. Instead of a homogenous Dirichlet boundary condition, we will formulate a free-slip boundary condition in the neighbourhood of the contact points $\Gamma \cap \partial \Omega$.  This way the velocity error around the contact points will be a direct result of the curvature error. The reasoning is that if the curvature errors do not influence the velocity too much, we can be sure that other type of boundary conditions involving the contact angle, e.g. the slip-condition proposed in \cite{Ren2007} and extensively studied in \cite{ZHANG2020109636}, can be applied without interference.

With this type of boundary condition for the velocity and since there is no outside force again, we know this problem has the same exact solution
\begin{align*}
&u = 0,
&p = \left\{\begin{array}{ll} c_0 + \kappa, & x\in \Omega_1 \\
         c_0 , & x\in \Omega_2\end{array}\right. 
\end{align*} with a constant $c_0 \in \mathbb{R}$.

\subsubsection{Orthogonal Interface Angle}
\label{OrthogonalWetting}

First we will simply choose 
$c = \begin{pmatrix}
0\\ -1
\end{pmatrix} \in \partial \Omega$ as the center of our circle while the radius $R = 0.5$ remains the same as in the previous section. This leads to an orthogonal angle between the interface $\Gamma$ and the boundary $\partial \Omega$ in the contact points.

\paragraph{Exact Interface}

In our first wetting test in Figure \ref{fig:WettingOrthogonalExact} for the exact interface we see a very similar behavior as in the case for the static droplet in Figure \ref{fig:CenterDropletExact} as in we need to cross a certain treshold of the regularization parameter $\alpha$ for the solution of our reconstruction to be smooth and accurate and thus reducing the errors in velocity and pressure as well.

\begin{figure}[h!]
  \centering
\begin{subfigure}{.5\textwidth}
  \centering
  \includegraphics[width=.9\linewidth]{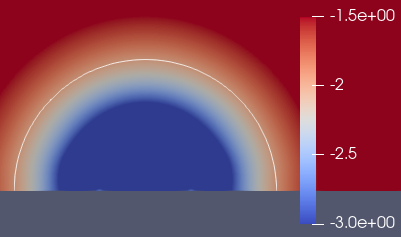}
  \caption{$\alpha = 10$}

\end{subfigure}%
 \hfill
\begin{subfigure}{.5\textwidth}
  \centering
  \includegraphics[width=.9\linewidth]{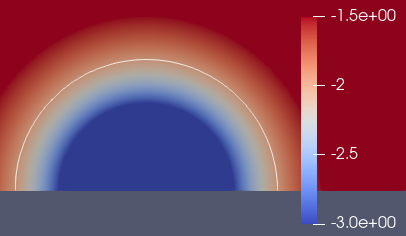}
  \caption{$\alpha = 0.01$}
\end{subfigure}%

\begin{subfigure}{.5\textwidth}
  \centering
  \includegraphics[width=.9\linewidth]{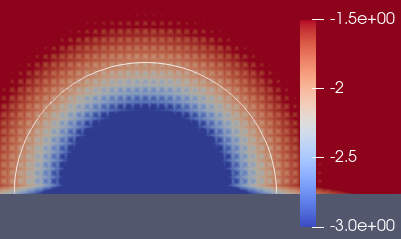}
  \caption{$\alpha = 0.0001$}
\end{subfigure}%
 \hfill
\begin{subfigure}{.5\textwidth}
  \centering
  \includegraphics[width=.9\linewidth]{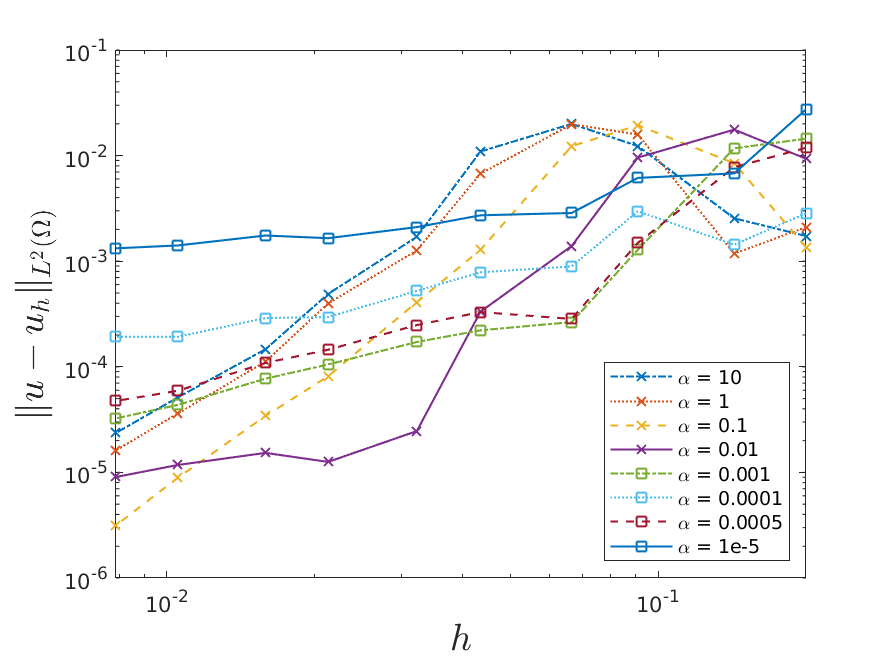}
  \caption{Velocity error}
\end{subfigure}%

\begin{subfigure}{.5\textwidth}
  \centering
  \includegraphics[width=.9\linewidth]{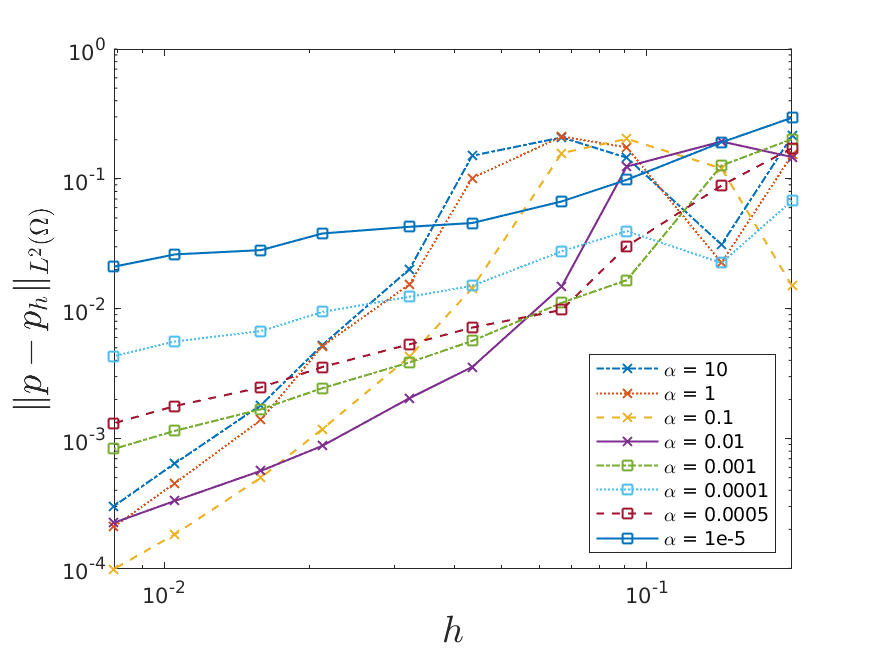}
  \caption{Pressure error}
\end{subfigure}%
 \hfill
\begin{subfigure}{.5\textwidth}
  \centering
  \includegraphics[width=.9\linewidth]{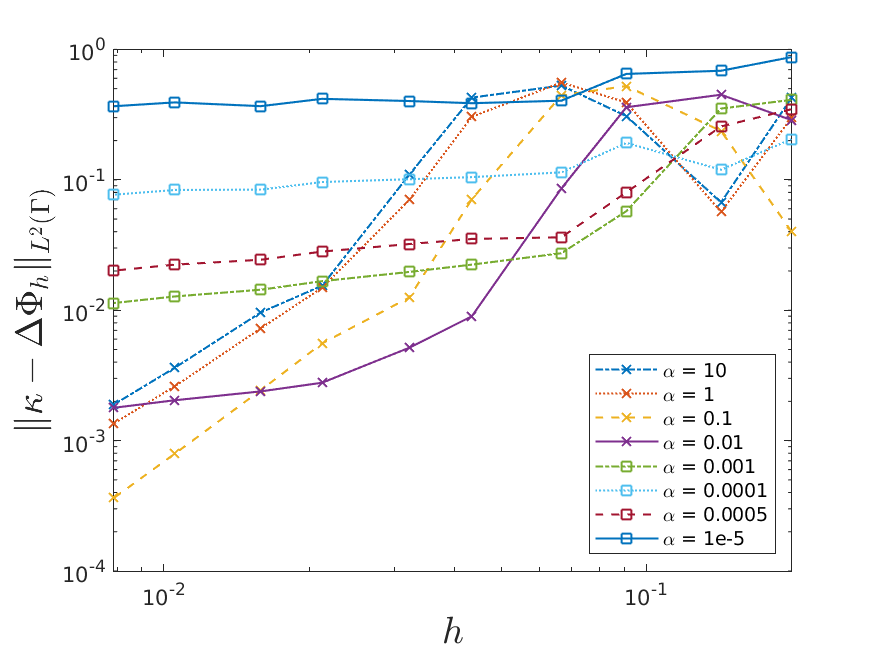}
  \caption{Curvature error}
\end{subfigure}%

\caption{Orthogonal wetting results for the exact interface function}
  \label{fig:WettingOrthogonalExact}
\end{figure}

\paragraph{Numerical signed distance function}

For the wetting scenario with the numerical signed distance function in Figure \ref{fig:WettingOrthogonalAllSigneed} we also see a similar outcome as for the static droplet test in Figure \ref{fig:CenterDropletAllSigned}. Especially we note that the results are better the larger the regularization parameter $\alpha$ is and the error behaves similarly for large $\alpha$.

\begin{figure}[h!]
  \centering
\begin{subfigure}{.5\textwidth}
  \centering
  \includegraphics[width=.9\linewidth]{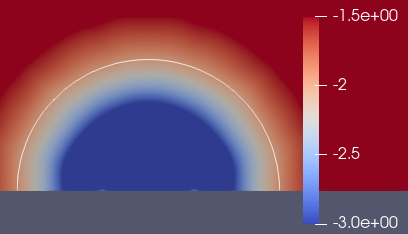}
  \caption{$\alpha = 10$}

\end{subfigure}%
 \hfill
\begin{subfigure}{.5\textwidth}
  \centering
  \includegraphics[width=.9\linewidth]{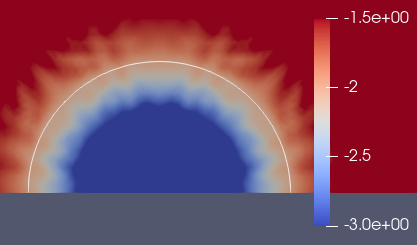}
  \caption{$\alpha = 0.01$}
\end{subfigure}%

\begin{subfigure}{.5\textwidth}
  \centering
  \includegraphics[width=.9\linewidth]{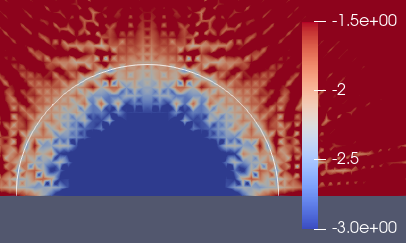}
  \caption{$\alpha = 0.0001$}
\end{subfigure}%
 \hfill
\begin{subfigure}{.5\textwidth}
  \centering
  \includegraphics[width=.9\linewidth]{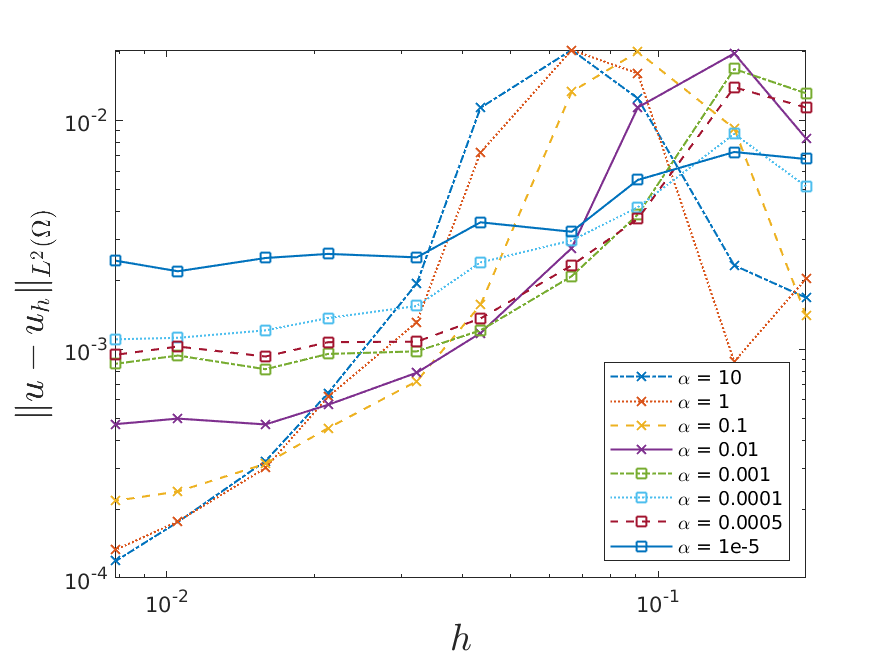}
  \caption{Velocity error}
\end{subfigure}%

\begin{subfigure}{.5\textwidth}
  \centering
  \includegraphics[width=.9\linewidth]{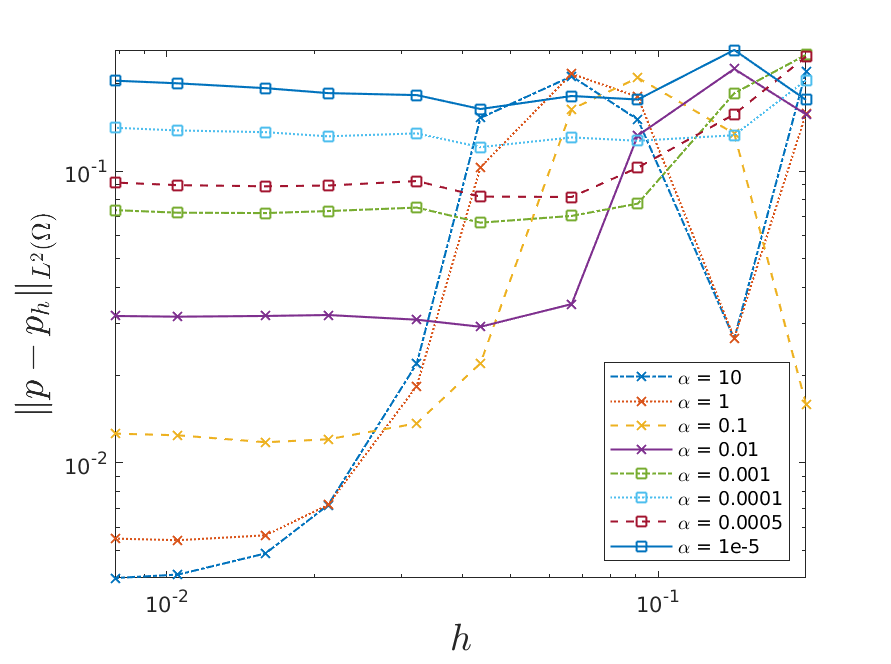}
  \caption{Pressure error}
\end{subfigure}%
 \hfill
\begin{subfigure}{.5\textwidth}
  \centering
  \includegraphics[width=.9\linewidth]{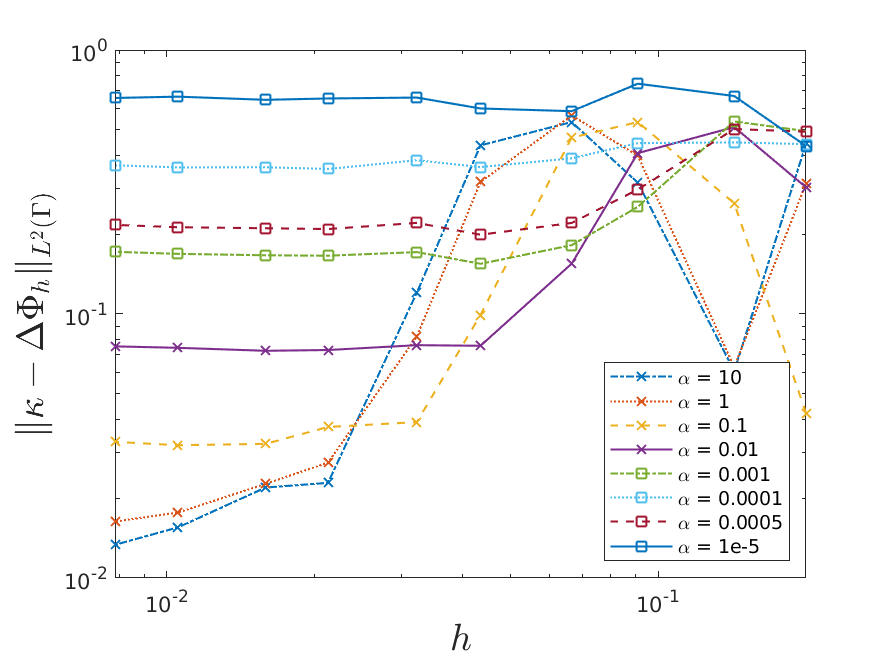}
  \caption{Curvature error}
\end{subfigure}%

\caption{Orthogonal wetting results for the numerical signed interface function}
  \label{fig:WettingOrthogonalAllSigneed}
\end{figure}

\clearpage
\paragraph{FMM level set function}

Now in the case of the FMM level set function, the results in Figure \ref{fig:WettingOrthogonalFMM} are again similar to the results of the center droplet case in Figure \ref{fig:CenterDropletFMM}. The error behavior and the influence of $\alpha$ on it mimic the static droplet case, although we note the errors are behaving a bit worse. 

\begin{figure}[h!]
  \centering
\begin{subfigure}{.5\textwidth}
  \centering
  \includegraphics[width=.9\linewidth]{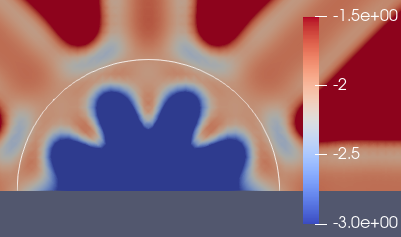}
  \caption{$\alpha = 10$}

\end{subfigure}%
 \hfill
\begin{subfigure}{.5\textwidth}
  \centering
  \includegraphics[width=.9\linewidth]{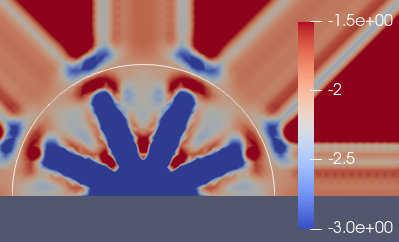}
  \caption{$\alpha = 0.01$}
\end{subfigure}%

\begin{subfigure}{.5\textwidth}
  \centering
  \includegraphics[width=.9\linewidth]{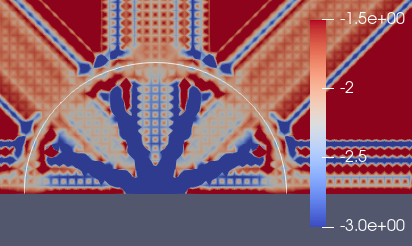}
  \caption{$\alpha = 0.0001$}
\end{subfigure}%
 \hfill
\begin{subfigure}{.5\textwidth}
  \centering
  \includegraphics[width=.9\linewidth]{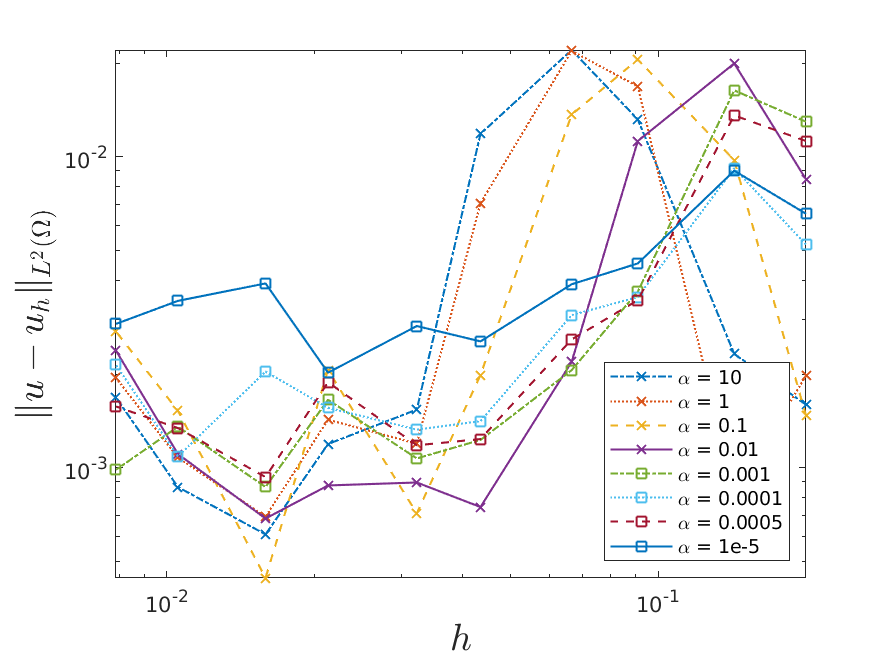}
  \caption{Velocity error}
\end{subfigure}%

\begin{subfigure}{.5\textwidth}
  \centering
  \includegraphics[width=.9\linewidth]{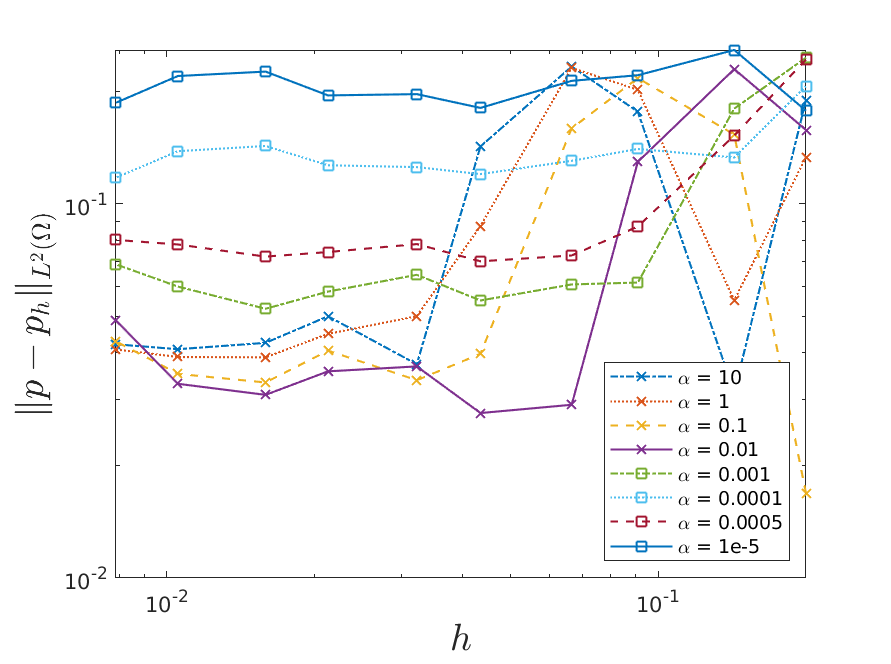}
  \caption{Pressure error}
\end{subfigure}%
 \hfill
\begin{subfigure}{.5\textwidth}
  \centering
  \includegraphics[width=.9\linewidth]{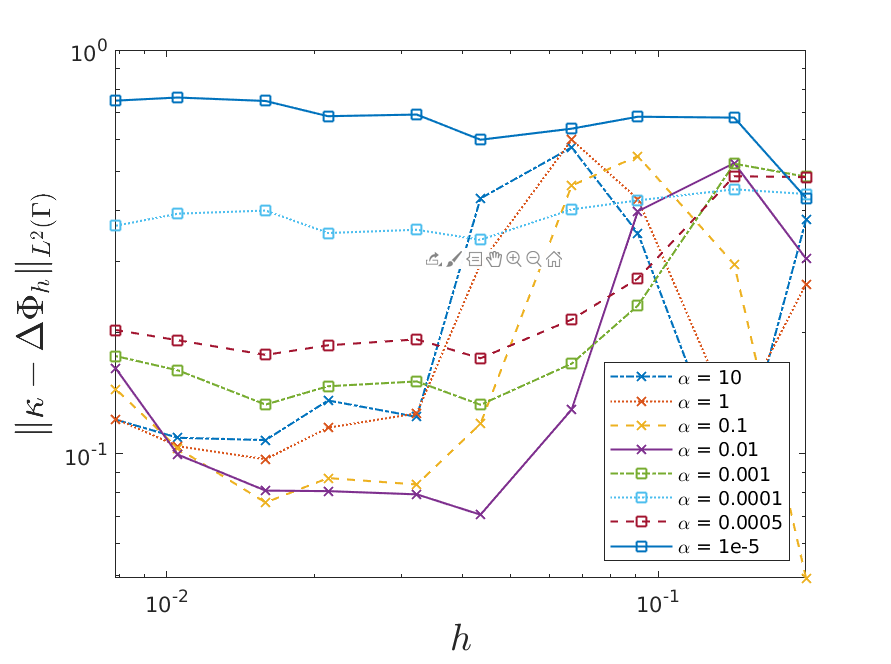}
  \caption{Curvature error}
\end{subfigure}%

\caption{Orthogonal wetting results for the FMM interface function}
  \label{fig:WettingOrthogonalFMM}
\end{figure}

\clearpage
\subsubsection{Sharp Interface Angle}

In this section we will now perform another series of wetting simulations as in Section \ref{OrthogonalWetting}  but with the center $c$ changed to $c = \begin{pmatrix}
0\\ -1.25
\end{pmatrix}$ such that the interface $\Gamma$ intersects the domain boundary $\partial \Omega$ in a sharper angle.

Considering different contact angles is beneficial to ensure that our reconstruction will later work in a full and time depedent wetting simulation.

\paragraph{Exact Interface}

First, in the case of the projected exact interface, we note that the error behavior in Figure \ref{fig:WettingSharpExact} does not resemble the error behavior of the corresponding static droplet test in Figure \ref{fig:CenterDropletExact} or the previous wetting test in Figure \ref{fig:WettingOrthogonalExact}. Notably, the error is now becoming worse for very large $\alpha$ and there seems to be a sweetspot for the regularization parameter where the error is minimal. This is akin to the previous Laplacian test case we have conducted and the observation is similar as well: With very large $\alpha$ our reconstruction is getting smoothed in an undesirable way. 

\begin{figure}[h!]
  \centering
\begin{subfigure}{.5\textwidth}
  \centering
  \includegraphics[width=.9\linewidth]{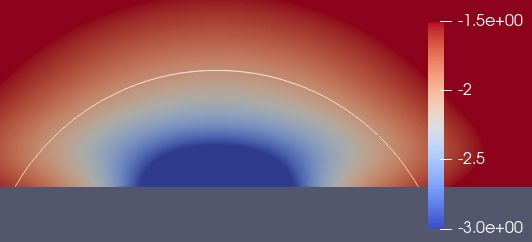}
  \caption{$\alpha = 10$}

\end{subfigure}%
 \hfill
\begin{subfigure}{.5\textwidth}
  \centering
  \includegraphics[width=.9\linewidth]{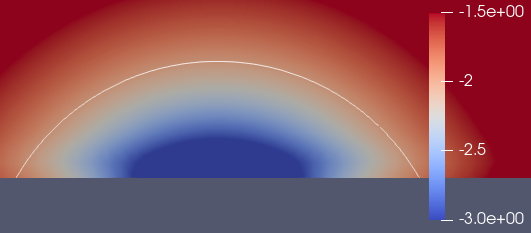}
  \caption{$\alpha = 0.01$}
\end{subfigure}%

\begin{subfigure}{.5\textwidth}
  \centering
  \includegraphics[width=.9\linewidth]{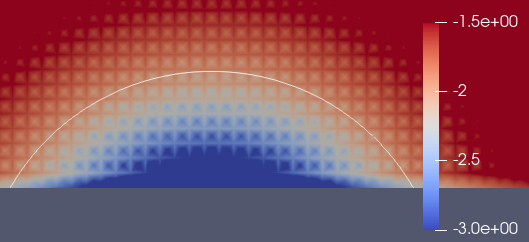}
  \caption{$\alpha = 0.0001$}
\end{subfigure}%
 \hfill
\begin{subfigure}{.5\textwidth}
  \centering
  \includegraphics[width=.9\linewidth]{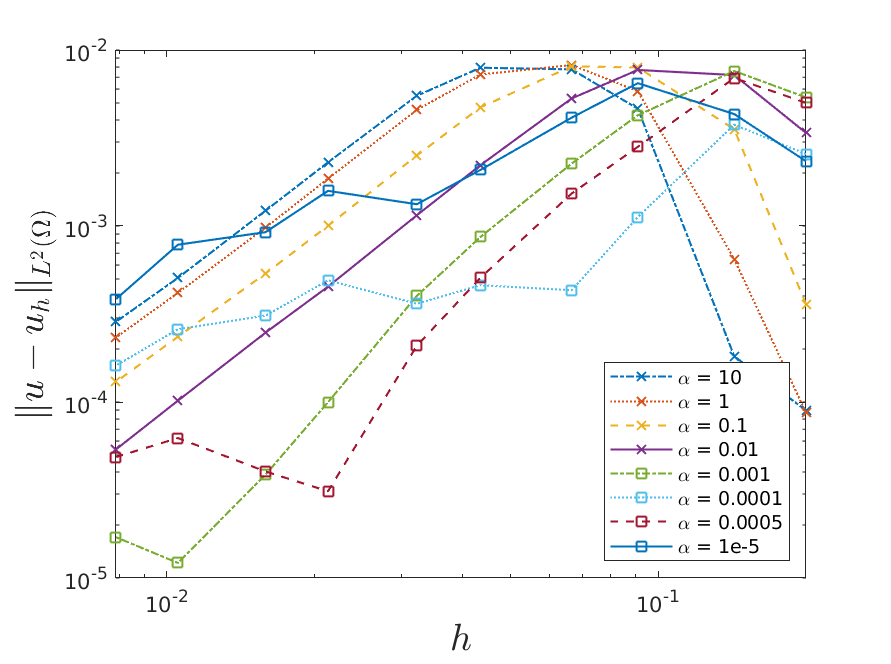}
  \caption{Velocity error}
\end{subfigure}%

\begin{subfigure}{.5\textwidth}
  \centering
  \includegraphics[width=.9\linewidth]{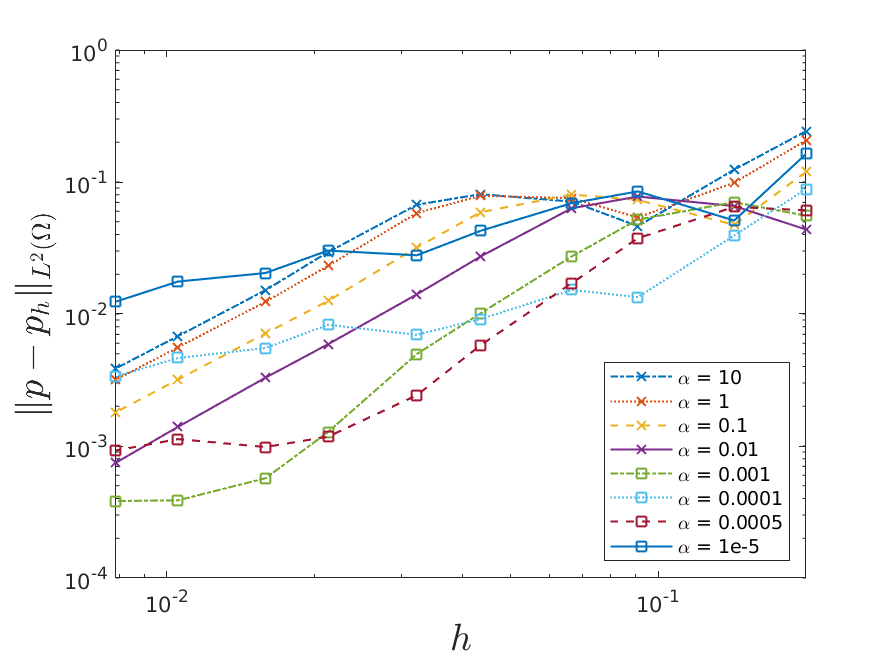}
  \caption{Pressure error}
\end{subfigure}%
 \hfill
\begin{subfigure}{.5\textwidth}
  \centering
  \includegraphics[width=.9\linewidth]{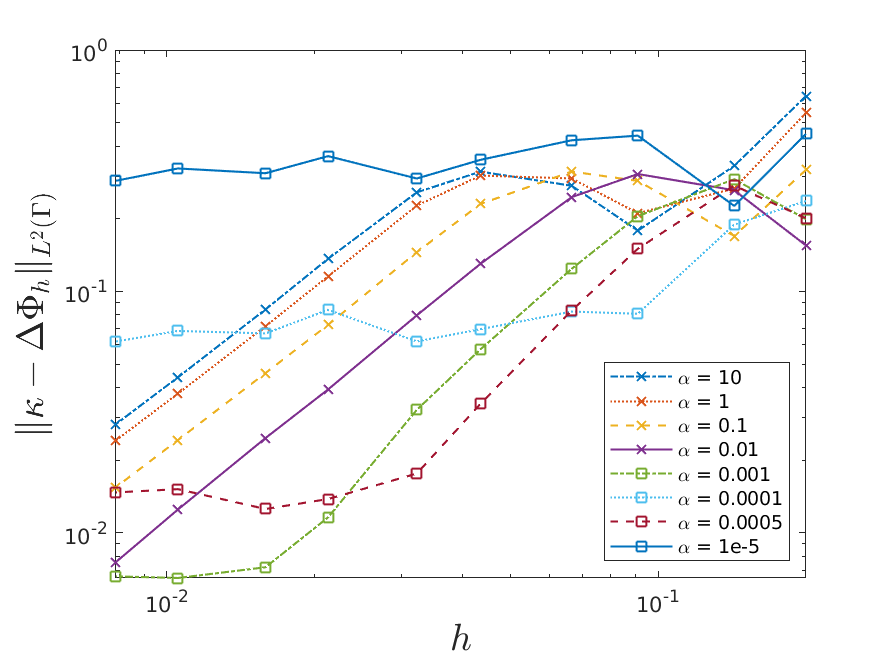}
  \caption{Curvature error}
\end{subfigure}%

\caption{Sharp wetting results for the exact interface function}
  \label{fig:WettingSharpExact}
\end{figure}

\clearpage
\paragraph{Comparison between $H^2$ and $H^3$ reconstruction}
Briefly going back to the $H^2$ reconstruction, we can see in Figure \ref{fig:H2vsH3WettingSharp} that the boundary conditions for the $H^2$-reconstruction negatively affect the curvature reconstruction devastatingly. As we can see in the error graphs, the curvature error even becomes larger when increasing the step size. As such we deem $H^2$-reconstruction unusable for our purpose of calculating an accurate curvature term.

\begin{figure}[h!]
  \centering
\begin{subfigure}{.5\textwidth}
  \centering
  \includegraphics[width=.9\linewidth]{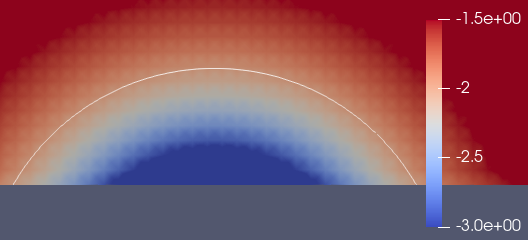}
  \caption{$H^3$ reconstruction with $\alpha = 0.0005$}

\end{subfigure}%
 \hfill
\begin{subfigure}{.5\textwidth}
  \centering
  \includegraphics[width=.9\linewidth]{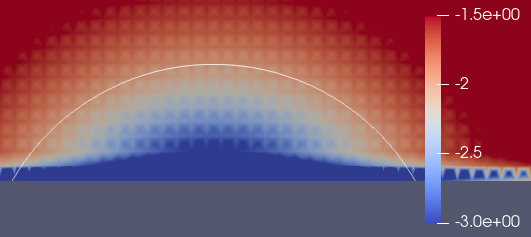}
  \caption{$H^2$ reconstruction with $\alpha = 0.01$}
\end{subfigure}%

\begin{subfigure}{.5\textwidth}
  \centering
  \includegraphics[width=.9\linewidth]{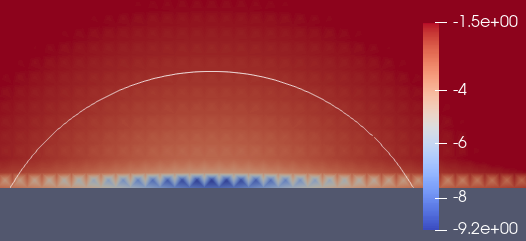}
  \caption{$H^2$ reconstruction with $\alpha = 0.01$, rescaled color bar}
\end{subfigure}%
 \hfill
\begin{subfigure}{.5\textwidth}
  \centering
  \includegraphics[width=.9\linewidth]{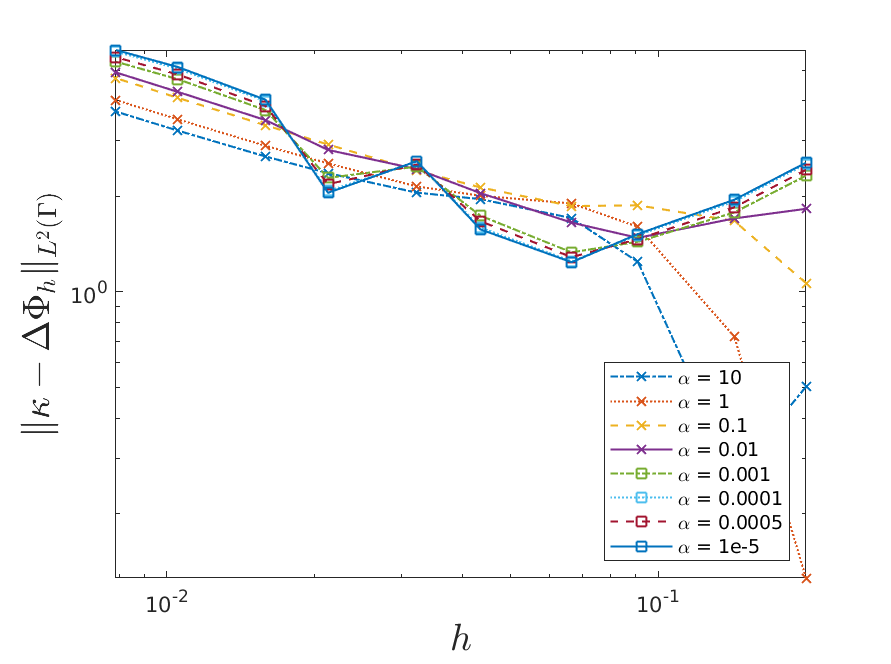}
  \caption{Error graphs for $H^2$ reconstruction}
  \end{subfigure}
    \caption{Comparison between $H^3$ and $H^2$ Tikhonov based reconstruction for the case of a sharp interface angle}
    \label{fig:H2vsH3WettingSharp}
\end{figure}

\clearpage
\paragraph{Numerical signed distance function and FMM level set function}

Finally we will apply our $H^3$-reconstruction to the more realistic examples of a numerical signed distance function and an FMM level set function. As we can see in Figures \ref{fig:WettingSharpAllSigneed} and \ref{fig:WettingSharpFMM} the resulting reconstructions are considerably worse around the contact point than in the case of the projected exact level set function. As we see in the error graphs, increasing the step size even has a negative outcome for the errors.

\begin{figure}[h!]
  \centering
\begin{subfigure}{.5\textwidth}
  \centering
  \includegraphics[width=.9\linewidth]{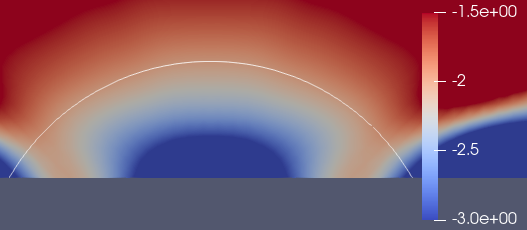}
  \caption{$\alpha = 10$}

\end{subfigure}%
 \hfill
\begin{subfigure}{.5\textwidth}
  \centering
  \includegraphics[width=.9\linewidth]{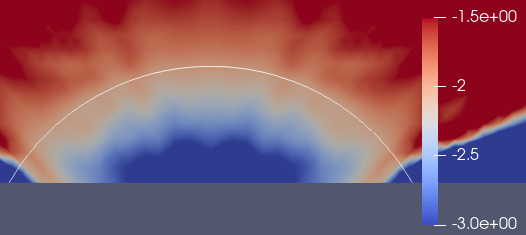}
  \caption{$\alpha = 0.01$}
\end{subfigure}%

\begin{subfigure}{.5\textwidth}
  \centering
  \includegraphics[width=.9\linewidth]{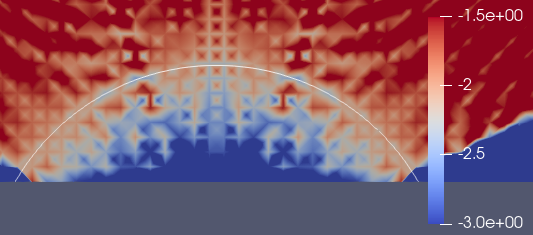}
  \caption{$\alpha = 0.0001$}
\end{subfigure}%
 \hfill
\begin{subfigure}{.5\textwidth}
  \centering
  \includegraphics[width=.9\linewidth]{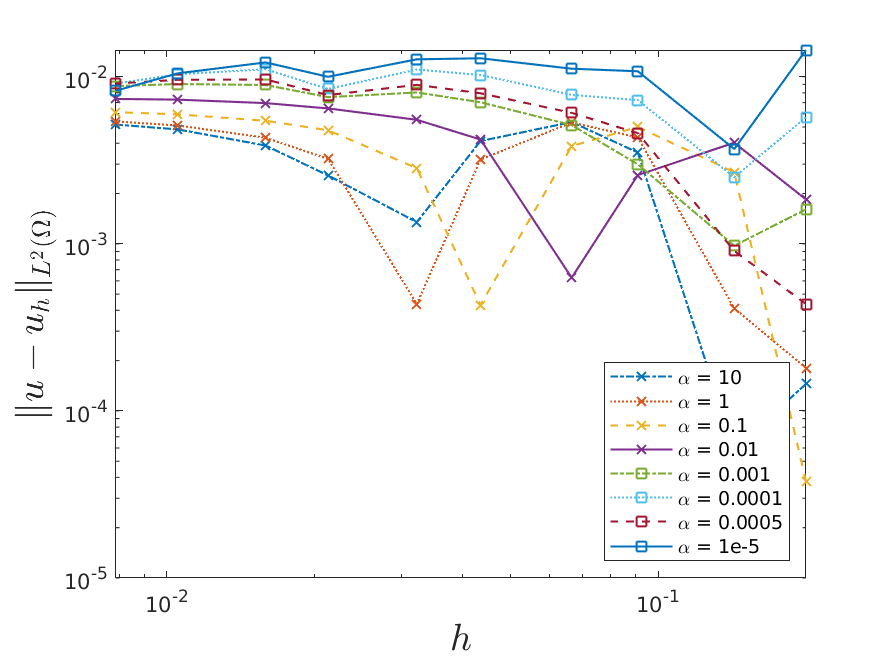}
  \caption{Velocity error}
\end{subfigure}%

\begin{subfigure}{.5\textwidth}
  \centering
  \includegraphics[width=.9\linewidth]{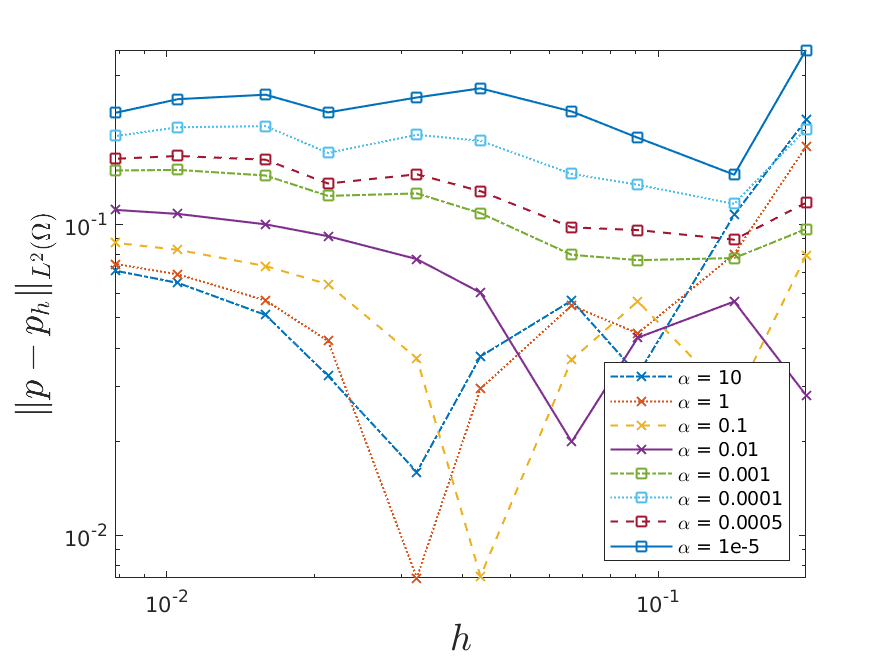}
  \caption{Pressure error}
\end{subfigure}%
 \hfill
\begin{subfigure}{.5\textwidth}
  \centering
  \includegraphics[width=.9\linewidth]{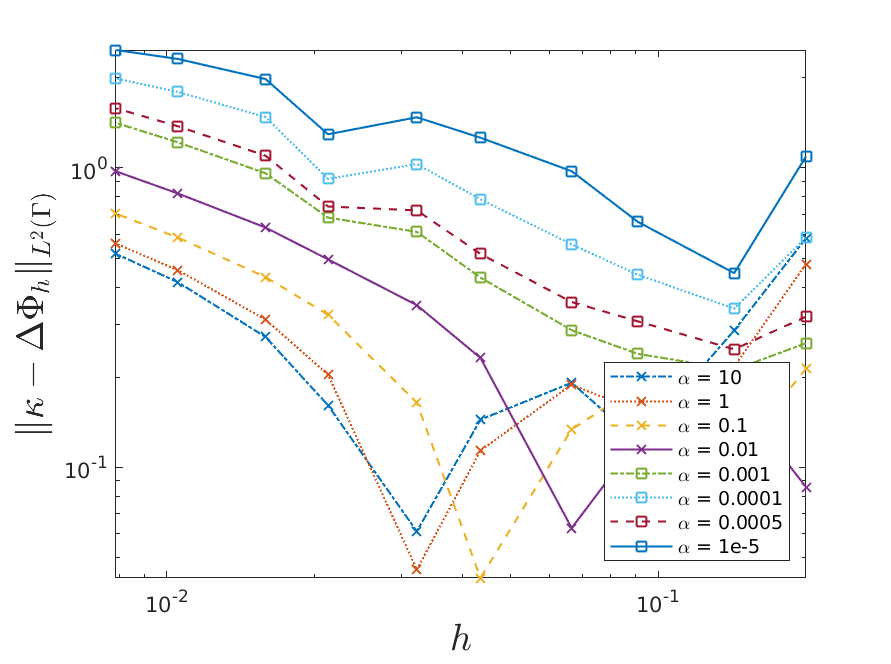}
  \caption{Curvature error}
\end{subfigure}%

\caption{Sharp wetting results for the numerical signed interface function}
  \label{fig:WettingSharpAllSigneed}
\end{figure}

\begin{figure}[h!]

  \centering
\begin{subfigure}{.9\textwidth}
  \centering
   \includegraphics[width=.8\linewidth]{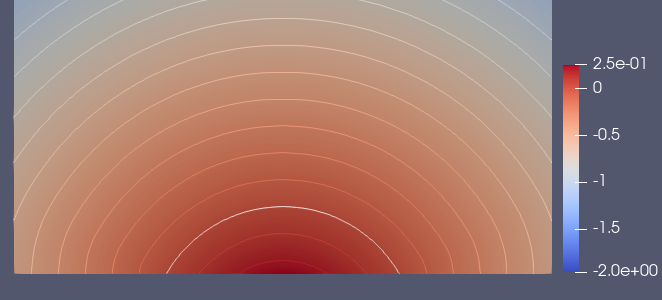}
  \caption{Isolines of numerical signed distance function for sharp wetting}

\end{subfigure}%
 \hfill
\begin{subfigure}{.9\textwidth}
  \centering
   \includegraphics[width=.8\linewidth]{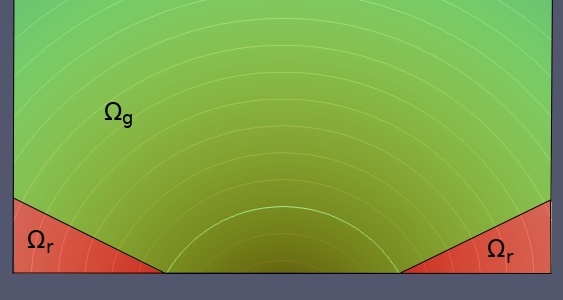}
  \caption{Domains in sharp wetting numerical signed distance function}
\end{subfigure}%

\caption{Isoline behavior for sharp wetting case}
    \label{fig:WettingSharpIsolines}
\end{figure}

Now the reason for this behavior comes from the way the level set functions are calculated numerically: Let $Circ_c:= \{x \in \mathbb{R}^2: |x-c|= 0.5 \}$ be the circle with radius $0.5$ around the center $c \in \mathbb{R}^2$. In this test setting we had chosen 
$c = \begin{pmatrix}
0\\ -1.25
\end{pmatrix}$. 

When we calculate the level set function numerically, we minimize the distances to $Circ_c \cap \Omega$, i.e. the part of the circle $Circ_c$ which lives on the domain $\Omega$. Obviously this is what we are supposed to do, as in a time dependent setting where the interface is part of the solution we will know the interface only on our domain. If we look at the isolines of the numerical signed distance function in Figure \ref{fig:WettingSharpIsolines}, we see that we have two regions. In the green region the nearest interface point for an $x\in \Omega_g$ is always found on the intersection between the interface $\Gamma_h$ and the line $\overline{xc}$.  

But in the red region $\Omega_r$, the nearest interface point is always the nearest contact point, i.e. the closer point where $\Gamma_h$ intersects $\partial\Omega$. So, in the green region the isolines of level set function $\phi_h$ consist of concentric circles around $c$ whereas in the red regions the isolines consist of concentric circles around the contact points $\Gamma_h \cap \partial\Omega$. This sudden change has a huge effect when reconstructing second derivatives. 

As such the 'exact' interface results in Figure \ref{fig:WettingSharpExact} actually did have a bias in them, as this projected level set function takes all points of the circle $Circ_c$ into account, i.e. interface points outside of the domain $\Omega$.

\begin{figure}[h!]
  \centering
\begin{subfigure}{.5\textwidth}
  \centering
  \includegraphics[width=.9\linewidth]{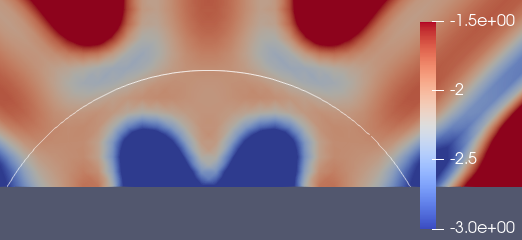}
  \caption{$\alpha = 10$}

\end{subfigure}%
 \hfill
\begin{subfigure}{.5\textwidth}
  \centering
  \includegraphics[width=.9\linewidth]{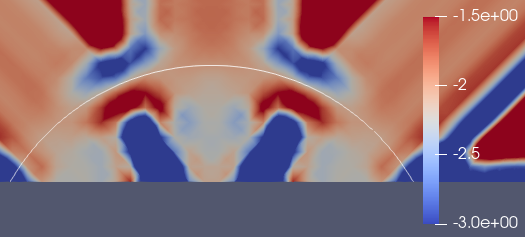}
  \caption{$\alpha = 0.01$}
\end{subfigure}%

\begin{subfigure}{.5\textwidth}
  \centering
  \includegraphics[width=.9\linewidth]{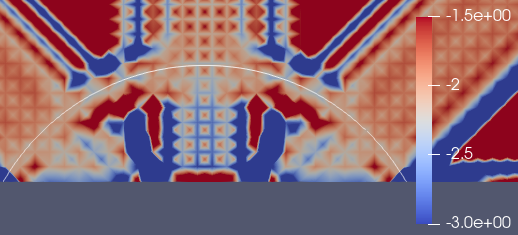}
  \caption{$\alpha = 0.0001$}
\end{subfigure}%
 \hfill
\begin{subfigure}{.5\textwidth}
  \centering
  \includegraphics[width=.9\linewidth]{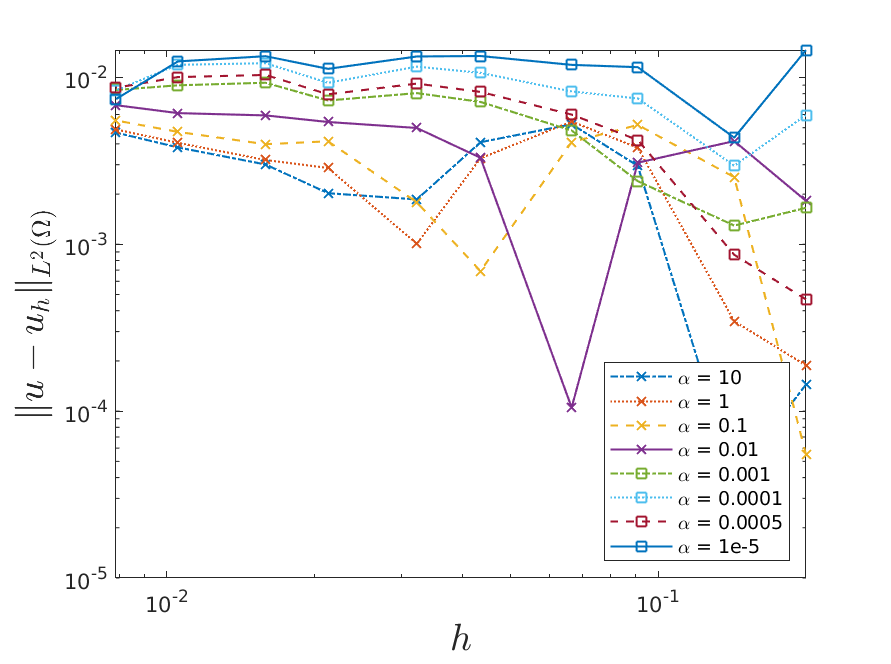}
  \caption{Velocity error}
\end{subfigure}%

\begin{subfigure}{.5\textwidth}
  \centering
  \includegraphics[width=.9\linewidth]{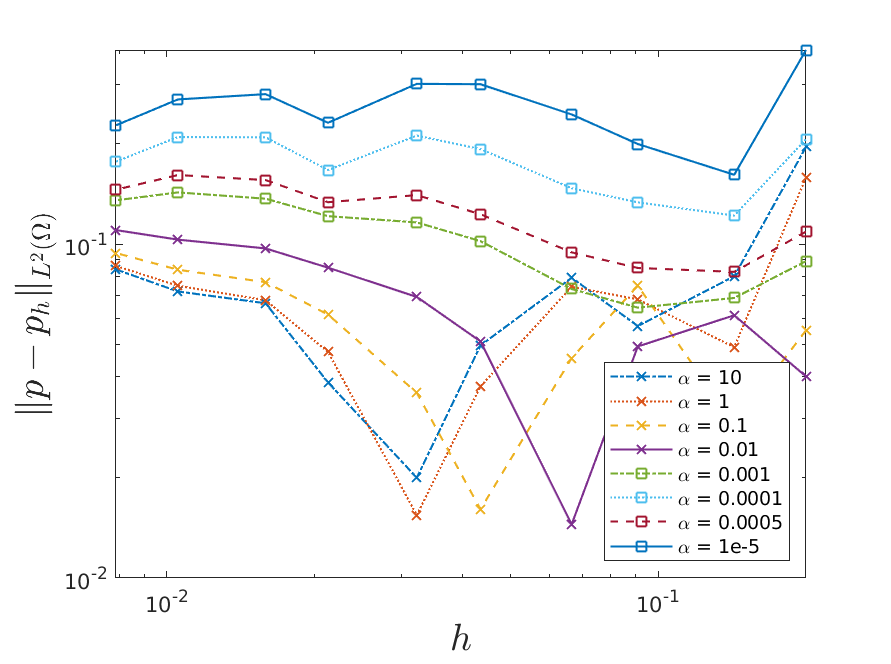}
  \caption{Pressure error}
\end{subfigure}%
 \hfill
\begin{subfigure}{.5\textwidth}
  \centering
  \includegraphics[width=.9\linewidth]{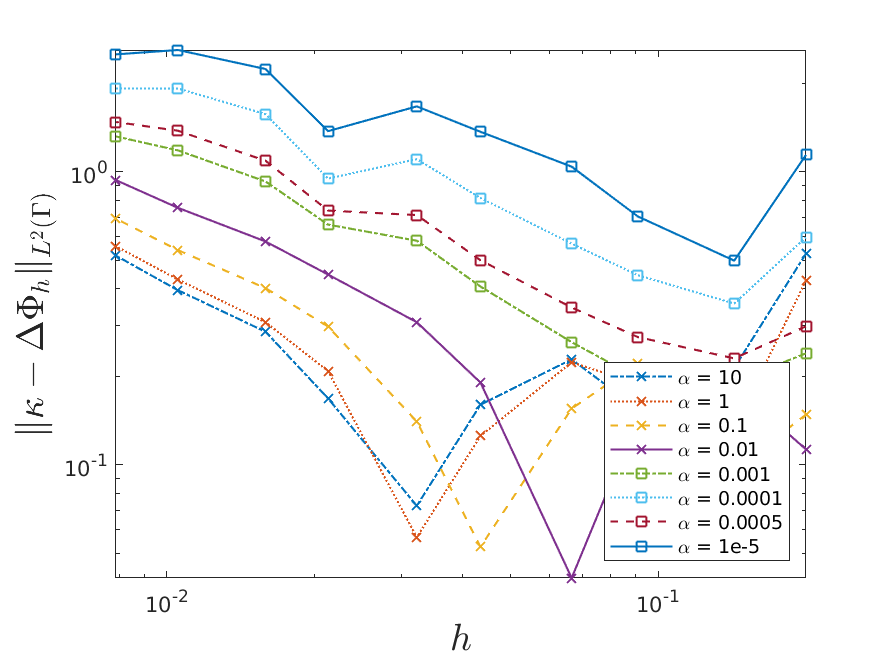}
  \caption{Curvature error}
\end{subfigure}%

\caption{Sharp wetting results for the FMM interface function}
  \label{fig:WettingSharpFMM}
\end{figure}

Conclusively, the problems we experience in this section are therefore not inherent to our method of curvature reconstruction but come from the input data of the level set function. One possible remedy could be to extend the numerical interface $\Gamma_h$ past the boundary $\partial\Omega$ in a sensible way and calculate the level set function based on this extended interface. 
But this will not be done in this paper anymore and shall be investigated in a follow up work.

\clearpage
\section{Conclusion and outlook}

In this article we have shown a method to calculate the unique minimum of a functional of the type
\begin{align*}
J(\Phi) = \|A\Phi - \phi_h \|_{L^2(\Omega)}^2 + \alpha \| \Phi \|^2_{H^k(\Omega)}.
\end{align*} by solving a PDE of order $2k$ in the cases $k =2,3$.
During numerical experiments, we have successfully shown this type of functional is suitable to find the Laplacian of a function where we only know its piecewise linear representation. Particularly, we have seen how both the $H^2$ and $H^3$ reconstruction can approximate the Laplacian correctly in the interior of the domain. The behavior of the reconstruction towards the domain boundary looked promising in the $H^3$ case but warrants further investigation with regards to the used FEM spaces.

Furthermore, we have successfully shown how this method can be applied to the case of calculating the surface tension functional when solving the two-phase stationary Stokes equation when the exact level set function describing the interface is known. In the case of a numerically calculated level set function we have identified additional problems in their construction which need to be solved before this type of method can be applied in solving time dependent two-phase Navier-Stokes equation.

Another possible future direction of this work is investigating analytical error behavior of our method and whether the analytical error analysis matches our observation of the errors during our conducted numerical research.

An additional point of interest lies in optimization of the computational cost of our method, which will need to be adressed in the future.

\section*{Acknowledgment}

This work was partially funded by the Deutsche Forschungsgemeinschaft (DFG, German Research Foundation) – 439916647.

\footnotesize
\bibliography{literature}
\end{document}